\documentclass[11pt,reqno]{amsart}
\usepackage{amsmath, amssymb, amsthm}
\usepackage[linktocpage=true]{hyperref}
\hypersetup{colorlinks,linkcolor=blue,urlcolor=blue,citecolor=blue}
\usepackage[table]{xcolor}
\usepackage{stmaryrd} 
\usepackage{bbm} 
\usepackage{leftidx} 
\usepackage{mathabx} 
\usepackage{rotating} 
\usepackage{wasysym} 
\usepackage[mmddyy]{datetime}
\usepackage{qtree}
\usepackage{float} 
\usepackage{arydshln}
\makeatletter
\renewcommand*\env@matrix[1][*\c@MaxMatrixCols c]{%
  \hskip -\arraycolsep
  \let\@ifnextchar\new@ifnextchar
  \array{#1}}
\makeatother
\makeatletter
\newcommand{\labeltarget}[1]{\Hy@raisedlink{\hypertarget{#1}{}}}
\makeatother
\usepackage{todonotes}
\usepackage{pgf, tikz, colortbl, braids}
\usetikzlibrary{arrows, positioning, calc, chains, cd, braids}
\tikzset{
	ch/.style={circle,draw,on chain,inner sep=2pt},
	chj/.style={ch,join},
	}

\numberwithin{equation}{subsection}
\usepackage[margin = 0.9in]{geometry}
\def\<{\langle}
\def\>{\rangle}

\newcommand{\B}{\mrm{B}}
\newcommand{\D}{\mrm{D}}

\newcommand{\Ba}[1]{    \begin{array}   #1   \end{array}}
\newcommand{\Bc}[1]{     \begin{cases}   #1   \end{cases}}

\newcommand{\cA}{\mathcal{A}}

\newcommand{\cH}{\mathcal{H}}

\newcommand{\crl}{\curvearrowleft}
\newcommand{\crr}{\curvearrowright}

\newcommand{\End}{\operatorname{End}}

\newcommand{\fgl}{\mathfrak{gl}}

\newcommand{\aff}{\mrm{aff}}
\newcommand{\Deg}{\mrm{deg}}
\newcommand{\ext}{\mrm{ext}}

\newcommand{\GL}{\mrm{GL}}
\newcommand{\Hom}{\mrm{Hom}}

\newcommand{\inv}{^{-1}}

\newcommand{\mrm}{\mathrm}

\newcommand{\Mor}{\underset{\mrm{Mor}}{\simeq}}

\newcommand{\NN}{\mathbb{N}}

\newcommand{\tif}{\textup{if }}

\newcommand{\Tr}{\mathop{\textup{tr}}}

\newcommand{\usubseteq}{\protect\rotatebox[origin=c]{90}{$\subseteq$}}
\newcommand{\dsubseteq}{\protect\rotatebox[origin=c]{270}{$\subseteq$}}

\newcommand{\wo}[1]{w_\circ(#1)}

\newcommand{\ZZ}{\mathbb{Z}}
\newcommand{\gt}{>}
\newcommand{\lt}{<}

\renewcommand{\=}[1]{\overline{#1}}

\usepackage{enumerate}
\newenvironment{eq}{\begin{equation}}{\end{equation}}

\theoremstyle{definition}

\newtheorem{Def}{Definition}[subsection]

\newtheorem{exa}[Def]{Example}
\newtheorem{rmk}[Def]{Remark}

\theoremstyle{plain}
\newtheorem{prop}[Def]{Proposition}
\newtheorem{thm}[Def]{Theorem}
\newtheorem{theorem}[Def]{Theorem}
\newtheorem{lemma}[Def]{Lemma}
\newtheorem{lem}[Def]{Lemma}
\newtheorem{cor}[Def]{Corollary}

\begin{document}

\title{Quantum wreath products and Schur--Weyl duality I}

\begin{abstract} In this paper the authors introduce a new notion called the quantum wreath product, 
which {is} the algebra $B \wr_Q \mathcal{H}(d)$ {produced} from a given algebra $B$, 
a positive integer $d$, and a choice $Q=(R,S,\rho,\sigma)$ of parameters. Important examples {that arise from our construction} include many variants of the Hecke algebras, such as the Ariki-Koike algebras, the affine Hecke algebras and their degenerate version, Wan-Wang's wreath Hecke algebras, Rosso-Savage's (affine) Frobenius Hecke algebras, Kleshchev-Muth's affine zigzag algebras, and the Hu algebra that quantizes the wreath product $\Sigma_m \wr \Sigma_2$ between symmetric groups. 

In the first part of the paper, the authors develop a structure theory for the quantum wreath products.
Necessary and sufficient conditions for these algebras to afford a basis of suitable size are obtained.
Furthermore, a Schur-Weyl duality is established via a splitting lemma and mild assumptions on the base algebra $B$.  
Our uniform approach encompasses many known results which were proved in a case by case manner. 
The second part of the paper involves the problem of constructing natural subalgebras of Hecke algebras that arise from wreath products. 
Moreover, a bar-invariant basis of the Hu algebra via an explicit formula for its extra generator is also described. 

\end{abstract}

\author{\sc Chun-Ju Lai}
\address{Institute of Mathematics \\ Academia Sinica \\ Taipei 106319, Taiwan} 
\email{cjlai@gate.sinica.edu.tw}
\thanks{Research of the first author was supported in part by
MSTC grants 113-2628-M-001-011, 112-2628-M-001-003, 111-2628-M-001-007, 109-2115-M-001-011, and the National Center of Theoretical Sciences.}

\author{\sc Daniel K. Nakano}
\address{Department of Mathematics\\ University of Georgia \\
Athens\\ GA~30602, USA}
\thanks{Research of the second author was supported in part by
NSF grant DMS-2101941}
\email{nakano@math.uga.edu}

\author{\sc Ziqing Xiang}
\address{ Department of Mathematics and National Center For Applied Mathematics Shenzhen\\ Southern University of Science and Technology \\Shenzhen 518055, China} 
\thanks{Research of the third author was  supported by NSFC (12350710787) and NSFC (12471311)}
\email{xiangzq@sustech.edu.cn}

\keywords{}
\subjclass{}

\maketitle


\section{Introduction}

\subsection{}\label{S:motivation} The Iwahori-Hecke algebra $\cH_q(W)$ is a $q$-deformation of the group algebra $KW$ of the Coxeter group $W$ which was later extended to a deformation of a complex reflection group $W$ (see \cite{BMR98}). In particular, it was established independently \cite{BM93, AK94} by {Brou\'e}-Malle and Ariki-Koike that a (quantum) algebra, called either the cyclotomic Hecke algebra or the Ariki-Koike algebra, deforms the complex reflection group $G(m,1,d) = C_m \wr \Sigma_d$ (i.e., the wreath product of a cyclic group $C_m$ by a symmetric group $\Sigma_d$). 
The Ariki-Koike algebra affords a theory that is closely connected to the affine Hecke algebra of type A. 

For finite groups, there is a well-established theory involving subgroups and their relationship to the ambient group. However, for Hecke algebras, there are very few constructions of natural subalgebras that 
arise within these algebras. For example, an open problem is to develop a Sylow theory for Hecke algebra when $q$ is specialized to a root of unity. A fundamental construction of Sylow subgroups for the 
symmetric group involves wreath products so a natural question to ask is whether one can make constructions of wreath products using Hecke algebras. The following problems will constitute the focus of this paper.

\begin{enumerate}[(A)]

\item Developing a theory of wreath products between (quantum) algebras that provides a uniform treatment to all Hecke-like algebras, including but not limited to the Ariki-Koike algebras, the affine Hecke algebras of type A and their (degenerate) variants, Rosso-Savage's (affine) Frobenius Hecke algebras, affine zigzag algebras, and the Hu algebra;

\item Constructing ``the Hecke algebra'' $\cH_{m\wr d}$ that arises from the wreath product $\Sigma_m \wr \Sigma_d$ between the symmetric groups in the sense that the chain of subgroups $\Sigma_m^d \subseteq \Sigma_m \wr \Sigma_d \subseteq \Sigma_{md}$ deforms to a chain of subalgebras $\cH_q(\Sigma_m^d) \leq \cH_{m \wr d} \leq \cH_q(\Sigma_{md})$, where $\cH_{m \wr d}$ affords a standard basis and a bar-invariant basis, both indexed by $\Sigma_m \wr \Sigma_d$.
\end{enumerate}

Throughout this paper, we consider algebras {that are free} over a commutative ring $K$. {All tensor products are considered to be over $K$ unless specified otherwise.}

\subsection{The Hu Algebra} Our paper was initially motivated by the work of Jun Hu (see \cite{Hu02}). Hu constructed an algebra $\cA(m)$ in order to state a Morita equivalence theorem between the Hecke algebras of type D and of type A (under an invertibility condition),
\eq \label{eq:HD}
\cH_q(W(\D_{2m})) \Mor  
\cA(m)\oplus \prod_{i=1}^{m-1} \cH_q(\Sigma_i) \otimes \cH_q(\Sigma_{2m-i}).
\endeq
In Hu's first construction, he considered an equation in the type B Hecke algebra of unequal parameters $(1,q)$ that involves the Jucys-Murphy elements, and 
proved that there exists a unique element $H_1$ (called $h_m^*$ therein) in the type A Hecke algebra $\cH_q(\Sigma_{2m})$ that solves the equation, without an explicit formula for $H_1$.
The  {\em Hu algebra} is defined as the subalgebra $\cA(m) \subseteq \cH_q(\Sigma_{2m})$ generated by $H_1$ as well as the parabolic subalgebra $\cH_q(\Sigma_m\times \Sigma_m)$. 
The Hu algebra should be regarded as the canonical Hecke algebra for $\Sigma_{m} \wr \Sigma_2$. Recall that $\Sigma_m \wr \Sigma_2 \subseteq \Sigma_{2m}$ is generated by  $\Sigma_m \times \Sigma_m$ together with a ``thickened braid'' $t_1 := w_{m,m}$ described as below: 
\eq
w_{m,m} = 
\raisebox{-0.8cm}{
\begin{tikzpicture}
\pic[
braid/.cd,
number of strands=6,
height = 0.25cm,
border height = 0.25cm,
gap=-0.1,
control factor=0,
nudge factor=0,
name prefix=braid,
] at (0,0) {braid={s_3 s_4-s_2 s_1-s_3-s_5 s_2-s_4 s_3}};
\node[at=(braid-4-e), above] {1};
\node[at=(braid-5-e), above] {$\dots$};
\node[at=(braid-6-e), above] {$m$};
\node[at=(braid-1-e), above] {$m+1$};
\node[at=(braid-2-e), above] {$\dots$};
\node[at=(braid-3-e), above] {$2m$};
\end{tikzpicture}
}
=
\raisebox{-0.8cm}{\begin{tikzpicture}
\pic[
braid/.cd,
number of strands=2,
height = 1cm,
width = 2cm,
border height = 0.3cm,
line width = 3pt,
gap=-0.1,
control factor=0,
nudge factor=0,
name prefix=braid,
] at (0,0) {braid={s_1}};
\node[at=(braid-2-e), above] {$[1,m]$};
\node[at=(braid-1-e), above] {$[1,m]+m$};
\end{tikzpicture}
}.
\endeq

It is proved that $H_1$ specializes to $2^m t_1$ at $q=1$.
Moreover, $\cA(m)$ admits a presentation with generators $H_1$ and $T_i$'s subject to the wreath relations $H_1 T_i = T_{i'} H_1$ where $i'$ is either $i+m$ or $i-m$, as well as a ``quadratic relation'' of the form $H_1^2 = z_{m,m}$ for some central element $z_{m,m} \in Z(\cH_q(\Sigma_m \times \Sigma_m))$.

The key observation for constructing a general wreath product is that one should accommodate quadratic relations of the form:
$H^2 = SH+ R$, where $R,S$ are determined by certain elements in $B \otimes B$ for some algebra $B$; in contrast, in the usual quadratic relation $H^2 = (q-1)H+q$, $R,S$ are always scalars in the base ring.
\subsection{Quantum Wreath Product} 
We address the main problems in Section~\ref{S:motivation} simultaneously by introducing a new notion which we call the {\em quantum wreath product}.
More specifically, starting from an algebra $B$, a given {integer $d \ge 2$}, and a choice $Q$ of parameters, our quantum wreath product produces an algebra $B \wr \cH(d) = B \wr_Q \cH(d)$ generated by (lower level) tensor algebra $B^{\otimes d}$ together with (upper level) Hecke-like generators $H_1, \dots, H_{d-1}$ (see Definition~\ref{def:Qwr}). Other salient features are listed below. 

\begin{enumerate}[(A)]
\item 
This procedure $B \mapsto B \wr_Q \cH(d)$ that produces an associative algebra can be thought of as a quantization of the wreath product $G\mapsto G \wr \Sigma_d$. 
In particular, as long as $B$ is a deformation of the group algebra of a finite group $G$, our quantum wreath product $B \wr \cH(d)$ is a deformation of the group algebra of $G \wr \Sigma_d$.
\item In general, $B$ does not need to be finite-dimensional nor commutative. For instance, in the Hu algebra case, the base algebra $B= \cH_q(\Sigma_d)$ is not commutative. 
In another case when $B = K[X^{\pm1}]$ is the Laurent polynomial {ring}, there is a quantum wreath product $K[X^{\pm1}] \wr \cH(d)$ that realizes the affine Hecke algebra of type A.
\item We are able to determine the necessary and sufficient conditions on the choice of parameters for $B \wr \cH(d)$ to admit bases of the ``right'' size (see Theorem~\ref{thm:basis}).
Our proof of the basis theorem relies on a key lemma having a similar flavor as Kashiwara's grand loop induction (see Lemma~\ref{lem:loop}).
\end{enumerate}
In \cite{E22}, Elias proved a variation of our basis theorem for his Hecke-type categories (see Section~\ref{sec:HC}, Section~\ref{sec:RelBasis}) in the following sense: 
(1) Elias' proof uses Bergman's diamond lemma for algebras over commutative rings, i.e., the statement is valid for those quantum wreath products whose base algebras are commutative. Thus, Elias' basis theorem does not imply our basis theorem. 
(2) His basis theorem is focused on describing the minimal set of ambiguities, which corresponds to the first step of our proof. In contrast, we proceed and then determine the conditions on the choice of parameters so that these ambiguities are resolvable. Such a determination of coefficients was not pursued in \cite{E22}.
(3) In contrast to Elias' proof which is based on the diamond lemma, we provide a Humphreys-type proof based on an abstract version of the regular representation.


In Table~\ref{tab:main} we provide a list of known examples that manifest themselves as quantum wreath products which are deformed from group or monoid algebras.
\begin{table}[!h]
\centering
\scalebox{0.8}{
\begin{tabular}{|l|l|l|}
\hline
Groups and monoids & Deformed algebras&$B\wr \cH(d)$
\\[0.1cm]
\hline
Wreath product $\Sigma_m \wr \Sigma_2$ & Hu algebra &$\cH_{m \wr 2} = \cA(m)$
\\
&& \quad$= \cH_q(\Sigma_m) \wr \cH(2)$
\\[0.1cm]
\hline
Complex reflection group $G(m,1,d) $
& Ariki-Koike algebra,  e.g.& $\cH_{m,d} = \frac{K[X^{\pm1}]\wr \cH(d)}{\langle(\prod_{i=1}^m (X-q_i))\otimes 1^{\otimes d-1}\rangle} $
\\
\quad e.g. Weyl group of type A/B/C
&
\quad Hecke algebra of type A/B/C&
\\
$G(m,1,d) = C_m\wr \Sigma_d$ & Yokonuma-Hecke algebra &$Y_{m,d} = KC_m \wr \cH(d)$
\\[0.1cm]
\hline
Extended affine symmetric group 
&Extended affine Hecke algebra &$\cH_q^\ext(\Sigma_d)$   
\\
$\ZZ \wr \Sigma_d = P \rtimes \Sigma_d$&&$\quad= K[X^{\pm1}]\wr \cH(d)$
\\
\quad$\usubseteq$ &\quad$\usubseteq$&\quad$\usubseteq$
\\
\quad  affine symmetric group  $\Sigma^\aff_d = Q \rtimes \Sigma_d$
&\quad  affine Hecke algebra   &$\cH_q(\Sigma^\aff_d)$
\\[0.1cm]
\hline
$\NN \wr \Sigma_d = P^+ \rtimes \Sigma_d$
&Degenerate affine Hecke algebra&$\cH^{\Deg}(\Sigma_d) = K[X]\wr \cH(d)$
\\[0.1cm]
\hline
Affine Weyl groups of type B/C&&
\\
$W(C^\aff_d) = W(\Sigma^\aff_2)\wr \Sigma_d \subseteq (\ZZ \wr \Sigma_2) \wr \Sigma_d$
&Affine Hecke algebra of type C &$\cH^\aff(C_d) \twoheadleftarrow K[X^{\pm1}, Y^{\pm1}]\wr \cH(d)$ 
\\
\quad$\usubseteq$ &&
\\
\quad $W^\ext(B_d) =  \ZZ^d \rtimes (\Sigma_2 \wr \Sigma_d)$
&\quad Extended AHA of type B
&$\cH^\ext(B_d) \subseteq K[X^\pm] \wr \cH(2d)$
\\[0.1cm]
\hline
Finite wreath $H:=G \wr \Sigma_d$ 
& Group algebra &$KH = KG \wr \cH(d)$
\\
& Wreath Hecke algebra & $\cH^\Deg(H) = K[X]G \wr \cH(d)$
\\[0.1cm]
\hline
\end{tabular}
}
\caption{Examples of quantum wreath products which deform from groups or monoids.}
\label{tab:main}
\end{table}

{In the literature,
Kleshchev-Muth first introduced in \cite{KM19} the {\em rank $n$ affinization} as a general approach to certain variants of degenerate affine Hecke algebras. 
Their main example are the affine zigzag algebras.
Savage introduced the {\em affine wreath algebras} in \cite{Sa20} as another unifying approach to some other variants of degenerate affine Hecke algebras,  including Wan-Wang's wreath Hecke algebras \cite{WW08}. 
Motivated by the endomorphism algebras in the (quantum) Frobenius Heisenberg categories, Rosso and Savage later developed in \cite{RS20} a quantum version which they call the {\em quantum affine wreath algebras} (or {\em affine Frobenius Hecke algebras}).
Their theory recovers the action of the quantum Heisenberg category of \cite{BSW20} on module categories for Ariki-Koike algebras.
Also, their theory applies to the affine Sergeev algebra (i.e., the degenerate affine
Hecke-Clifford algebra) and Evseev-Kleshchev's super wreath product algebra appearing in their proof \cite{EK17} of Turner's conjecture.
}

There have been papers written by Banica, Bichon, Lemeux, and Tarrago  on ``free wreath product quantum groups”. Our construction of the quantum wreath product differs from this earlier notion. These authors considered the wreath product of a Hopf algebra with a quantum permutation group. Our quantum wreath product is the wreath product of an associative algebra with Hecke type algebras. 

\subsection{Additional Features of Quantum Wreath Products} 
We provide a brief summary of some additional significant properties of our quantum wreath product. 
\begin{enumerate}[(A)]
\item
{
Our theory not only provides a natural generalization of the theory of Rosso-Savage, and of Kleshchev-Muth (see Section~\ref{ex:RS}--\ref{ex:KM}), 
but also has the advantage of being able to treat a multitude of Hecke-like algebras and their affinizations (degenerate or not) as a result of taking a single quantum wreath product.
In particular, our general definition of quantum wreath product encompasses the three definitions in \cite{Sa20, RS20} for their unifying theory.
Furthermore, our theory applies to the Hu algebras, which are not covered by previous theory due to its intriguing quadratic relation.
}

\item
{
Another important feature of having a quantum wreath product realization is that the Schur-Weyl duality (i.e., double centralizer property) almost come for free,
in the sense that the construction of $B \wr \cH(d)$ already encodes necessary information for a Schur duality.
By combining Theorem~\ref{thm:SchurEquiv}, Proposition~\ref{prop:struc}, and Lemma~\ref{lem:QSchur}, we obtain a machinery to establish a Schur duality for $B \wr \cH(d)$ in the following sense, as long as the base algebra $B$ is symmetric:
}
\[
\begin{tikzcd}
\textup{Splitting for }B \arrow[rr, "\textup{Additional assumptions}", "\textup{in Lemma~\ref{lem:QSchur}}"', Rightarrow] \arrow[d, Rightarrow] &&\textup{Splitting for }A:=B\wr \cH(d)  \arrow[d, "\textup{Theorem~\ref{thm:SchurEquiv}}", Rightarrow]
\\
\textup{Double centralizer property for }B && \textup{Double centralizer property for }A
\end{tikzcd}
\]
{
Note that  Theorem~\ref{thm:SchurEquiv} also applies in the following scenario, which turns out to be useful for the study of the Hu algebras:} 
\[
\begin{tikzcd}
\textup{Splitting for }A \arrow[rr, "\textup{Additional assumptions}", "\textup{in Lemma~\ref{lem:split}}"', Rightarrow] \arrow[d, Rightarrow] &&\textup{Splitting for }A' \subseteq A  \arrow[d, "\textup{Theorem~\ref{thm:SchurEquiv}}", Rightarrow]
\\
\textup{Double centralizer property for }A && \textup{Double centralizer property for }A'
\end{tikzcd}
\]
{Furthermore, we obtain a Schur functor from the existence of the splitting, see Section \ref{sec:SF}.}
\end{enumerate}


\subsection{Bar-Invariant Basis Arising from $\Sigma_m \wr \Sigma_2$} 
In this paper, we obtain a new explicit formula for $H_1$ as a sum of $2^m$ terms in which every term is a quantization of $t_1 := w_{m,m}$ after normalization. 
For example, 
\eq
H_1(1) = T_1 + T_1^{-1} 
=\raisebox{-0.25cm}{
\begin{tikzpicture}
\pic[
braid/.cd,
number of strands=2,
height = 0.3cm,
width = 0.6cm,
border height = 0.3cm,
gap=0.1,
control factor=0,
nudge factor=0,
name prefix=braid,
] at (0,0) {braid={s_1}};
\node[at=(braid-2-e), above] {1};
\node[at=(braid-1-e), above] {2};
\end{tikzpicture} 
}
+
\raisebox{-0.25cm}{
\begin{tikzpicture}
\pic[
braid/.cd,
number of strands=2,
height = 0.3cm,
width = 0.6cm,
border height = 0.3cm,
gap=0.1,
control factor=0,
nudge factor=0,
name prefix=braid,
] at (0,0) {braid={s_1^{-1}}};
\node[at=(braid-2-e), above] {1};
\node[at=(braid-1-e), above] {2};
\end{tikzpicture} 
}
,
\endeq
and
\eq
H_1(2) = 
\raisebox{-0.5cm}{
\begin{tikzpicture}
\pic[
braid/.cd,
number of strands=4,
height = 0.25cm,
width = 0.6cm,
border height = 0.25cm,
gap=0.15,
control factor=0,
nudge factor=0,
name prefix=braid,
] at (0,0) {braid={s_2 s_3 -s_1 s_2}};
\node[at=(braid-3-e), above] {1};
\node[at=(braid-4-e), above] {$2$};
\node[at=(braid-1-e), above] {$3$};
\node[at=(braid-2-e), above] {$4$};
\end{tikzpicture}}
+
\raisebox{-0.5cm}{
\begin{tikzpicture}
\pic[
braid/.cd,
number of strands=4,
height = 0.25cm,
width = 0.6cm,
border height = 0.25cm,
gap=0.15,
control factor=0,
nudge factor=0,
name prefix=braid,
] at (0,0) {braid={s_2 s_3 -s_1^{-1} s_2^{-1}}};
\node[at=(braid-3-e), above] {1};
\node[at=(braid-4-e), above] {$2$};
\node[at=(braid-1-e), above] {$3$};
\node[at=(braid-2-e), above] {$4$};
\end{tikzpicture}}
+
\left(
\raisebox{-0.5cm}{
\begin{tikzpicture}
\pic[
braid/.cd,
number of strands=4,
height = 0.25cm,
width = 0.6cm,
border height = 0.25cm,
gap=0.15,
control factor=0,
nudge factor=0,
name prefix=braid,
] at (0,0) {braid={s_2^{-1} s_3^{-1}-s_1 s_2}};
\node[at=(braid-3-e), above] {1};
\node[at=(braid-4-e), above] {$2$};
\node[at=(braid-1-e), above] {$3$};
\node[at=(braid-2-e), above] {$4$};
\end{tikzpicture}
}
+
\raisebox{-0.5cm}{
\begin{tikzpicture}
\pic[
braid/.cd,
number of strands=4,
height = 0.25cm,
width = 0.6cm,
border height = 0.25cm,
gap=0.15,
control factor=0,
nudge factor=0,
name prefix=braid,
] at (0,0) {braid={s_2^{-1} s_3^{-1}-s_1^{-1} s_2^{-1}}};
\node[at=(braid-3-e), above] {1};
\node[at=(braid-4-e), above] {$2$};
\node[at=(braid-1-e), above] {$3$};
\node[at=(braid-2-e), above] {$4$};
\end{tikzpicture}}
\right) 
.
\raisebox{-0.5cm}{
\begin{tikzpicture}
\pic[
braid/.cd,
number of strands=4,
height = 0.25cm,
width = 0.6cm,
border height = 0.25cm,
gap=0.15,
control factor=0,
nudge factor=0,
name prefix=braid,
] at (0,0) {braid={s_3 s_3}};
\node[at=(braid-3-e), above] {3};
\node[at=(braid-4-e), above] {$4$};
\node[at=(braid-1-e), above] {$1$};
\node[at=(braid-2-e), above] {$2$};
\end{tikzpicture}}
.
\endeq
As a consequence, we discovered that $\cA(m)$ is closed under the bar-involution, has a bar-invariant basis, in which a positivity with respect to the dual canonical basis is observed. 
In contrast, it is not known whether groups of the form $\Sigma_m \wr \Sigma_2$ afford a canonical basis theory since they are generally not complex reflection groups.

Moreover, with the explicit formula for $H_1$, we are able to construct the generalized Hu algebra denoted by $\cH_{m \wr d}$. 
It is not clear at this time whether these algebras can be written as a quantum wreath product for $d \geq 3$. 


\subsection{}
In the sequel, we will use the representation theory of $\cH_{m\wr 2}$ to solve the type D analog of a conjecture by Ginzburg-Guay-Opdam-Rouquier in \cite{GGOR03}.
This problem entails constructing a quasi-hereditary 1-cover for Hecke algebras of type $D_{n}$ via the Schur functor for a module category of the Schur algebra.
When $n$ is odd, the solution is very similar to the type B approach developed in \cite{LNX22}, in light of a Morita equivalence theorem in \cite{Pa94};
for type $D_{2m} = G(2,2,2m)$. Our approach relies on constructing a 1-cover for  the Hu algebra $\cA(m)$ via the corresponding Schur algebra.
We expect that a similar theory holds for type $G(d,d,dm)$, in light of Hu-Mathas' generalization of $\cA(m)$ to a more general complex reflection group in \cite{HM12}.
\endrmk
\subsection{} 
The paper is organized as follows. 
In Section \ref{sec:pre} we recall the preliminaries, including important examples of the wreath product of groups.
In Section \ref{sec:Wreath}, we introduce the definition of our quantum wreath product;
while in Section \ref{sec:examples} we identify various Hecke-like algebras as quantum wreath products (see the Appendix \ref{sec:app} for definitions of these algebras). 
Section \ref{sec:RepA} is devoted to the basis theorem. In particular, we determine the necessary and sufficient conditions for a quantum wreath product to admit a basis of the ``right'' size.
See Appendix \ref{sec:loopproof} for details of the proof.
In Section \ref{sec:Hu}, we obtain a new explicit formula for the special generator of the Hu algebra, which leads to a bar-invariant basis, as well as a Hecke subalgebra which quantizes the wreath product between symmetric groups.
In Section \ref{sec:SchurW}, we develop a Schur-Weyl duality for quantum wreath product  $A = B \wr \cH(d)$ from those for the (symmetric) algebra $B$.
\vskip .25cm
\noindent {\bf Acknowledgements.}
We acknowledge Valentin Buciumas, Jun Hu, Hankyung Ko, Eric Marberg, Andrew Mathas, Catharina Stroppel and Weiqiang Wang,  for insightful discussions regarding various details of the paper.
We thank Alistair Savage for making us aware of the connections to the (affine) Frobenius Hecke algebras as well as for providing several suggestions on an early version of the paper.
{We thank Ben Elias for bringing to our attention connections with the Bergman diamond lemma for Hecke-type categories.}
\section{Preliminaries}\label{sec:pre}
\subsection{Wreath Products}\label{sec:wr}
Let $\Sigma_d$ be the symmetric group on the set $[d] := \{1, 2, \ldots, d\}$ generated by $S = \{s_1, \ldots, s_{d-1}\}$ consisting of simple transpositions $s_i = (i, i+1)$.
Define 
\eq\label{def:stos}
s_{a\to b}
=
\begin{cases}
s_a s_{a-1} \cdots s_{b+1} s_b &\tif a> b;
\\
{s_a }&{\tif a=b;}
\\
s_a s_{a+1} \cdots s_{b-1} s_b &\tif a < b,
\end{cases}
\quad
{
s_{a\to b \to a}
=
\begin{cases}
s_a s_{a-1} \cdots s_{b+1} s_b s_{b+1} \cdots s_a&\tif a > b;
\\
{s_a }&{\tif a=b;}
\\
s_a s_{a+1} \cdots s_{b-1} s_b s_{b-1} \cdots s_a&\tif a < b.
\end{cases}}
\endeq
Recall that for a group $G$, the wreath product $G \wr \Sigma_d$ is the semi-direct product $G^d \rtimes \Sigma_d$ whose multiplication rule is determined by 
\eq
(g_1, \ldots, g_d)w = w(g_{w(1)}, \ldots, g_{w(d)})
\quad
\textup{for all}
\quad
w \in \Sigma_d, \ g_i \in G.
\endeq
{When $G$ is finite, say $G \subseteq \Sigma_m$ for some $m$, then the direct product $G^d \subseteq \Sigma_m^d \subseteq \Sigma_{md}$. 
Thus, $G \wr \Sigma_d$ can be regarded as a subgroup of $\Sigma_{md}$
by identifying each $w \in \Sigma_d$ with the following permutation:
\eq\label{def:wrembed}
(i-1)m + j \mapsto (w(i)-1)m+j
\quad
\textup{for all} 
\quad
1\leq i \leq d, 1\leq j \leq m.
\endeq}

%

\exa[Generalized Symmetric Groups]
Let $C_m$ be the cyclic group of order $m$. The generalized symmetric group  $C_m \wr \Sigma_d$ identifies with the complex reflection group of type $G(m,1,d)$.
For a fixed generator $x \in C_m$, $C_m \wr \Sigma_d$ is generated by $x$ and $S$ subject to the relations in $\Sigma_d$, together with that $x^m = 1$ and $s_{1} x s_{1} x = x s_{1} x s_{1}$.

In particular,  the hyperoctahedral group $C_2 \wr \Sigma_d$ can be identified either with the signed permutation group on $[\pm d] := [d] \sqcup -[d]$, or with the type B Weyl group $W(\B_d)$ generated by $S \cup \{s_0\}$ with extra braid relations $(s_0 s_1)^2 = (s_1s_0)^2$ and $s_0 s_i = s_i s_0$ for $i>1$. 
\endexa

The complex reflection group $G(m,1,d)$ affords two non-isomorphic deformations -- the Ariki-Koike algebras (or cyclotomic Hecke algebras, see Section~\ref{sec:AK}), and the Yokonuma-Hecke algebras. 
\exa[$\Sigma_m \wr \Sigma_2$ and Weyl Groups of Type D]
The wreath product $\Sigma_m \wr \Sigma_2$ is now identified with the subgroup of $\Sigma_{2m}$ generated by the Young subgroup $\Sigma_m^2 := \<s_{(i-1)m+1}, \ldots, s_{im-1}~|~1\leq i\leq 2\> \subseteq \Sigma_{2m}$, as well as the element $t_1=w_{m,m} \in \Sigma_{2m}$ {(which corresponds to the generator of $\Sigma_2$ via \eqref{def:wrembed})}, where $w_{a,b}\in \Sigma_{a+b}$ is now the element given by,  in two-line notation,
\eq\label{def:wab}
w_{a,b} = \left(
\begin{array}{cccccc}
1 &\cdots& a &a+1& \cdots& a+b
\\
b+1& \cdots& b+a &1& \cdots &b
\end{array}
\right). 
\endeq 
It is well-known that $w_{a,b} = s_{a\to1} s_{a+1 \to2} \dots s_{a+b-1 \to b}$ is a reduced expression, and that
\eq
w_{a,b} s_i= \begin{cases}
s_{i+a} w_{a,b} &\tif i<b; \\
s_{i-b} w_{a,b} &\tif i>b .
\end{cases}
\endeq
\endexa
Consequently, we can treat both $\Sigma_m \wr \Sigma_2$ and $\Sigma_{2m}$ as subgroups of a larger wreath product, i.e., the signed permutation group on $[\pm 2m]$, and then obtain the following tower of groups:
\eq\label{eq:towerG}
\Sigma_m \times \Sigma_m \subseteq \Sigma_m \wr \Sigma_2 
\subseteq \Sigma_{2m} \subseteq W(\B_{2m}) = \Sigma_2 \wr \Sigma_{2m}.
\endeq
Finally, we denote by $W(\D_r)$ the subgroup of $W(\B_r)$ generated by $s_1, \ldots, s_{r-1}$ and $s_u := s_0s_1s_0$. Indeed, $W(\D_r)$ is a Weyl group of type $\D_r$ with the following Dynkin diagram:
 \vspace{-10pt}
\begin{figure}[H]
\label{figure:Dyn}
\centering
\begin{tikzpicture}
\matrix [column sep={0.6cm}, row sep={0.3 cm,between origins}, nodes={draw = none,  inner sep = 2pt}]
{
	&\node(U2)[draw, circle, fill=white, scale=0.6, label =$u$] {};
\\
&&&&
\\
	\node(L1) [draw, circle, fill=white, scale=0.6, label = below:1] {};
	&\node(L2)[draw, circle, fill=white, scale=0.6, label =below:2] {};
	&\node(L3) {$\cdots$};
	&\node(L4)[draw, circle, fill=white, scale=0.6, label =below:$r-1$] {};
\\
};
\begin{scope}
\draw (L2) -- node  {} (U2);
\draw (L1) -- node  {} (L2);
\draw (L2) -- node  {} (L3);
\draw (L3) -- node  {} (L4);
\end{scope}
\end{tikzpicture}
\end{figure}
 \vspace{-15pt}
\subsection{Affine Symmetric Groups}\label{sec:wr2}
Within this section, we allow $G$ to be an infinite monoid in a wreath product $G \wr \Sigma_d$.
Now we consider the weight lattice $P = \sum_{i=1}^d \ZZ \epsilon_i$ for $\GL_d$. The lattice $P$ contains the root lattice $Q = \sum_{i=1}^{d-1} \ZZ \alpha_i$, where $\alpha_i = \epsilon_{i} - \epsilon_{i+1}$. 
We further let $P^+ := \sum_{i=1}^d \NN \epsilon_i$.
Denote the corresponding translation monoids by $t^L =\{ t^\lambda ~|~ \lambda \in L\}$ for $L\in \{P, P^+, Q\}$, with $t^\lambda t^\mu = t^{\lambda+\mu}$ for all $\lambda, \mu \in L$.

The extended and affine symmetric groups are the semidirect products 
$\Sigma^\aff_d := t^Q \rtimes \Sigma_d$ and
$\Sigma^\ext_d := t^P \rtimes \Sigma_d$, respectively, whose multiplications are uniquely determined by
$
t^\lambda s_i = s_i t^{s_i(\lambda)}.
$
Next, under the identification $P \equiv \ZZ^d, \lambda = \sum \lambda_i \epsilon_i \mapsto (\lambda_1, \dots, \lambda_d)$, we see that the simple reflection $s_i$ on $P$ agrees with the place permutation $s_i$ on $\ZZ^d$:
\eq
\begin{split}
s_i(\lambda) &= \lambda - \<\lambda, \alpha_i\>\alpha_i
= \lambda - (\lambda_i - \lambda_{i+1})(\epsilon_i - \epsilon_{i+1})
\\
&= (\lambda_1, \dots, \lambda_{i-1},\lambda_{i+1}, \lambda_i,\lambda_{i+2}, \dots, \lambda_d).
\end{split}
\endeq
Therefore, $\Sigma^\ext_d$ can be realized as the wreath product $\ZZ \wr \Sigma_d$ whose multiplication is determined uniquely by
$
(\lambda_1, \dots, \lambda_d)s_i = s_i(\lambda_1, \dots, \lambda_{i-1},\lambda_{i+1}, \lambda_i,\lambda_{i+2}, \dots, \lambda_d)$;
while $\Sigma^\aff_d$ is a subgroup given by
\eq
\Sigma^\aff_d = \{(\lambda_1, \dots, \lambda_d)w \in \ZZ\wr\Sigma_d ~|~ \Sigma_i \lambda_i = 0\} \subsetneq \ZZ \wr \Sigma_d.
\endeq
Note that we can thus regard the extended affine Hecke algebra (or degenerate affine Hecke algebra, resp.) of type A as a quantization of the wreath product $\ZZ \wr \Sigma_d = \Sigma^\ext_d$ (or the submonoid $\NN \wr \Sigma_d \subseteq \ZZ \wr \Sigma_d$, resp.) as in Table \ref{tab:main}.
\rmk
$\Sigma^\aff_d$ is a Coxeter group with generators $s_0, \dots, s_d$ and the following Dynkin diagram:
 
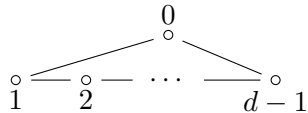
\vspace{-15pt}
\begin{figure}[H]
\label{figure:DynAffA}
\centering
\begin{tikzpicture}
\matrix [column sep={0.6cm}, row sep={0.3 cm,between origins}, nodes={draw = none,  inner sep = 2pt}]
{
	&&\node(U3)[draw, circle, fill=white, scale=0.6, label =$0$] {};
\\
&&&&
\\
	\node(L1) [draw, circle, fill=white, scale=0.6, label = below:1] {};
	&\node(L2)[draw, circle, fill=white, scale=0.6, label =below:2] {};
	&\node(L3) {$\cdots$};
	&\node(L4)[draw, circle, fill=white, scale=0.6, label =below:$d-1$] {};
\\
};
\begin{scope}
\draw (L1) -- node  {} (U3);
\draw (L1) -- node  {} (L2);
\draw (L2) -- node  {} (L3);
\draw (L3) -- node  {} (L4);
\draw (U3) -- node  {} (L4);
\end{scope}
\end{tikzpicture}
\end{figure}
 \vspace{-15pt}
 
The extra generator $s_0$ is obtained as follows:
let $\theta := \alpha_1 + \dots + \alpha_{d-1} = \epsilon_1 - \epsilon_d$ be the highest root. Then
\eq
s_0 = t^\theta s_\theta = (1, 0, \dots, 0, -1) s_\theta = s_\theta (-1, 0, \dots, 0, 1) = s_\theta t^{-\theta}.
\endeq
\endrmk
\subsection{Iwahori-Hecke Algebras}\label{sec:Hceke}
{Let $q\in K^\times$. For any Coxeter system $(W,S)$, the Hecke algebra $\cH_q(W)$ is the associative $K$-algebra generated by $T_s (s\in S)$, subject to the braid relations for $(W,S)$ as well as the quadratic relations $T_s^2  = (q-1)T_s + q$ for $s\in S$.
For the Weyl group $W(B_n)$ of type B, we will also consider its multi-parameter variant $\cH_{(Q,q)}(W(\B_{2m}))$ subject to the same braid relations, but with quadratic relations  $T_i^2 = (q-1) T_i + q$ (for $i>0$) and $T_0^2 = (Q-1)T_0 + Q$, where $Q\in K^\times$.
}

\subsection{Augmented Algebras}\label{sec:Aug}
Following \cite{GK93}, by an augmented algebra we mean a $K$-algebra $B$ equipped with a counit $\epsilon: B\to K$.
Denote by  $B^+ := \ker \epsilon \unlhd B$ the augmentation ideal of $B$.
Assume that $B \leq A$ is a normal subalgebra (i.e., $AB^+ = B^+A$). 
{Then one can form the quotient algebra (which is also augmented), \cite[5.2]{GK93}}, 
\eq
A//B := A/B^+A .
\endeq

\section{Quantum Wreath Products}\label{sec:Wreath}

Given a $K$-algebra $B$, {and an integer} $d \geq 2$, and a choice of parameter $Q=(R,S,\rho,\sigma)$, we introduce a construction called the {\em quantum wreath product} that produces an $K$-algebra  $B \wr \cH(d) = B \wr_Q \cH(d)$. Recall that a typical element in a wreath product group $G \wr \Sigma_d$ is of the form $(g_1, \dots, g_d)w$ for some $w\in \Sigma_d$, $g_i \in G$. 
We define an analogous algebra whose typical elements are linear combinations of elements the form below: 
\[
(b_1 \otimes \dots \otimes b_d)H_w
\quad
\textup{for some}
\quad
b_i \in B,
w\in \Sigma_d.
\]
\subsection{The Definition} 
For an element $Z \in B\otimes B$ and for $1\leq i \leq d-1$, set
\eq\label{def:Xi}
{Z_i} := \underset{i-1 \textup{ factors}}{\underbrace{1 \otimes \dots \otimes 1}} \otimes Z \otimes  \underset{d-i-1 \textup{ factors}}{\underbrace{1 \otimes \dots \otimes 1}} \in B^{\otimes d}.
\endeq 
Moreover, for an endomorphism $\phi \in \End_K(B\otimes B)$, set
\eq\label{def:Yi}
\phi_i(b_1 \otimes \dots \otimes b_d) := b_1 \otimes \dots \otimes b_{i-1}\otimes  \phi(b_{i} \otimes b_{i+1}) \otimes b_{i+2} \otimes \dots \otimes b_{d}.
\endeq 

\Def[Quantum Wreath Product]\label{def:Qwr}
Let $B$ be an associative $K$-algebra, and let $d\in \ZZ_{\geq 2}$.
Let $Q = (R,S,\rho,\sigma)$ be a choice of parameters with 
$R, S \in B \otimes B$, and 
$\rho, \sigma \in \End_K(B\otimes B)$ such that $\sigma$ is an automorphism.
The {\em quantum wreath product} is the associative $K$-algebra generated by the algebra $B^{\otimes d}$ and $H_1, \dots, H_{d-1}$ such that the following conditions hold, for $1\leq k \leq d-2, 1\leq i \leq d-1, |j-i|\geq 2$:
\begin{align}
&\textup{(braid relations) } \label{def:BR}
& H_k H_{k+ 1} H_k = H_{k+ 1} H_k H_{k + 1},  \quad H_i H_j = H_j H_i,
\\
&\textup{(quadratic relations) }\label{def:QR2}
&H_i^2 = {S_i}H_i + {R_i},
\\
&\label{def:WR2}\textup{(wreath relations) }
& H_ib = \sigma_i(b)H_i + \rho_i(b) \quad ( b\in B^{\otimes d}).
\end{align}
\noindent Often times, we refer this algebra  as $B \wr \cH(d)$, or $B \wr_Q \cH(d)$ whenever it is convenient.
\endDef

Here, ${R_i}$ is an element  in $B^{\otimes d}$, in contrast to those endomorphisms $\textup{R}_{i,i+1}$ in the literature acting on the $i$th and $(i+1)$th factors, in the context of the Yang-Baxter equations.

The quadratic and wreath relations determine the following local relations in $B\otimes B$, in the sense that one can recover \eqref{def:QR2}--\eqref{def:WR2} by embedding the local relation into the $i$th and $(i+1)$th positions of $B^{\otimes d}$ as well as replacing the symbol $H$ by $H_i$: 
\eq\label{eq:localrel}
H^2 = SH + R,
\quad
H (a\otimes b) = \sigma(a\otimes b) H + \rho(a \otimes b),
\quad
(a,b \in B).
\endeq


\subsection{Necessary Conditions} \label{sec:parameter}
We first describe certain conditions  on $Q$: 
%
\begin{align}
&\label{def:wr1}\sigma(1\otimes 1) = 1\otimes 1, 
\quad \rho(1 \otimes1)=0,
\\
&\label{def:wr2} \sigma(ab) = \sigma(a)\sigma(b),
\quad
\rho(ab) =  \sigma(a)\rho(b) + \rho(a)b,
\\
&\label{def:TTT}
\sigma(S)S + \rho(S) + \sigma(R) = S^2+R,
\quad
\rho(R) + \sigma(S)R = SR,
\\
&\label{def:qu1}
r_S\sigma^2+\rho\sigma+\sigma\rho = {l}_{S}\sigma,
\quad
r_R\sigma^2+\rho^2 = {l}_S\rho + {l}_R,
\end{align}
where  ${l_X}, r_X$ for $X\in B\otimes B$ are $K$-endomorphisms defined by  {left and} right multiplication in $B\otimes B$ by $X$, respectively. {The next five conditions will only be necessary when $d\geq 3$:}
\begin{align}
&\label{def:br1}\sigma_i\sigma_j\sigma_i = \sigma_j \sigma_i \sigma_j,
\quad \rho_i\sigma_j\sigma_i = \sigma_j \sigma_i \rho_j,
\\
&\label{def:br2}
\rho_i\sigma_j\rho_i = r_{S_j} \sigma_j \rho_i \sigma_j + \rho_j \rho_i \sigma_j + \sigma_j \rho_i \rho_j,
\\
&\label{def:br3}\rho_i\rho_j\rho_i + r_{R_i} \sigma_i \rho_j \sigma_i
= \rho_j\rho_i\rho_j + r_{R_j} \sigma_j \rho_i \sigma_j,
\end{align}
where $\{i,j\}= \{1,2\}$, $r_X$ for $X\in B^{\otimes 3}$ is understood as right multiplication in $B^{\otimes 3}$ by $X$. 
\begin{align}
\label{def:br4}&S_i = \sigma_j\sigma_i(S_j), \quad R_i = \sigma_j\sigma_i(R_j),
\quad
\rho_j\sigma_i(S_j) = 0 = \rho_j\sigma_i(R_j),
\\
\label{def:br5}& \sigma_j\rho_i(S_j) S_j +\rho_j\rho_i(S_j) + \sigma_j\rho_i(R_j) = 0 =  \rho_j\rho_i(R_j) + \sigma_j\rho_i(S_j) R_j,
\end{align}
where $\{i,j\}= \{1,2\}$. 
\endDef

Let $\{b_i \}_{i\in I}$ be a basis of $B$ for some index set $I$, and let $H_w := H_{i_1} \dots H_{i_N} \in B\wr \cH(d)$ for a reduced expression $w = s_{i_1} \dots s_{i_N}$.
Such an element $H_w$ is well-defined due to the braid relations \eqref{def:BR}.

\prop\label{prop:nec}
Suppose that $B\wr \cH(d)$ has {a} basis of the form $\{ (b_{\lambda_1} \otimes \dots\otimes b_{\lambda_d}) H_w ~|~ \lambda_j \in I, w\in \Sigma_d\}$.
Then, {Conditions \eqref{def:wr1}--\eqref{def:qu1} hold. 
Moreover, if $d\geq 3$, Conditions \eqref{def:br1}--\eqref{def:br5} hold.}
\endprop
\proof
The proposition follows from a direct calculation. See Appendix~\ref{sec:loopproof}.
\endproof
The conditions can be simplified when we put reasonable assumptions on $Q$. 
For example, if we assume that $\sigma$ is the flip map $a\otimes b \mapsto b\otimes a$,  
and that $\rho$ is the zero map, then Conditions \eqref{def:wr1}--\eqref{def:br5} are equivalent to the following conditions:
\eq\label{eq:simQ}
R = \sigma(R),
\quad
(\sigma(S)-S)R = 0,
\quad
{r_S = l_S \sigma}.
\endeq
\endrmk
Note that the more involved conditions such as \eqref{def:br2}--\eqref{def:br5} do not simplify much when $\rho \neq 0$ even if $\sigma$ is the flip map.
\rmk
{Equation \eqref{def:wr2}  is equivalent to the statement that $\sigma$ is an algebra automorphism, 
and that $\rho$ is a $\sigma$-derivation.  Then the wreath relation \eqref{def:WR2} appears in the definition of an Ore extension $B^{\otimes 2}[H_1; \sigma, \rho]$ as the new multiplication rule for the polynomial ring $B^{\otimes 2}[H_1]$. In other words, $B \wr \cH(2)$ is the quotient of this skew polynomial ring $B^{\otimes 2}[H_1; \sigma, \rho]$ by the quadratic relation $H_1^2 = SH_1 + R$.}
\endrmk
\subsection{Structure Theory} 
Now we state the basis theorem regarding the sufficient and necessary conditions on $Q$  such that $B \wr \cH(d)$ affords desirable bases. This 
theorem will be proved in Section~\ref{sec:RepA}. 
\begin{thm}\label{thm:basis}
Let $\{b_i\}_{i\in I}$ be a basis of $B$ for some index set $I$. 
The following are equivalent:
\begin{itemize}
\item[(a)]  {Conditions \eqref{def:wr1} -- \eqref{def:qu1} hold, and \eqref{def:br1} -- \eqref{def:br5} hold additionally if $d\geq 3$,}
\item[(b)] $\{ (b_{\lambda_1} \otimes \dots\otimes b_{\lambda_d}) H_w ~|~ \lambda_j \in I, w\in \Sigma_d\}$ forms a basis of $B \wr \cH(d)$,
\item[(c)]  $\{H_w (b_{\lambda_1} \otimes \dots\otimes b_{\lambda_d})  ~|~ \lambda_j \in I, w\in \Sigma_d\}$ forms a basis of $B \wr \cH(d)$.
\end{itemize}
\end{thm}

\subsection{} 
Consider a quantum wreath product $B \wr \cH(d)$.
Assume additionally that $B$ is an augmented algebra with counit $\epsilon: B \to K$,
{$R$ is invertible},
and that there is an $h_i \in K$ such that $h_i^2 = \epsilon(S) h_i + \epsilon(R)$ for each $i$.
{The space $B^{\otimes d}$ is also an augmented algebra, whose counit is also denoted by $\epsilon$.
It is natural to find the conditions on $\sigma$ and $\rho$ such that
}
 $B\wr \cH(d)$ is an augmented algebra whose counit is given by
\eq
\epsilon : B\wr \cH(d) \to K,
\quad
b_1 \otimes \dots  \otimes b_d \mapsto \epsilon(b_1) \dots \epsilon(b_d),
\quad
H_i \mapsto h_i. 
\endeq
{The counit is well-defined if and only if $\epsilon(H_i H_{i+1} H_i) = \epsilon (H_i H_{i+1} H_i)$ and $\epsilon(H_i b) = \epsilon (\sigma_i(b) H_{i} + \rho_i(b))$ for all $i$ and $b\in B^{\otimes d}$. 
Equivalently, $h_i$'s are all the same since $R$ (and hence $h_i$'s) are invertible. 
We can write $h := h_i$. Moreover,
\eq\label{eq:augwreath}
h(\epsilon(b) - \epsilon\sigma(b)) = \epsilon\rho (b) \quad \textup{for all } b \in B\otimes B.
\endeq
Hence, $B \wr \cH(d)$ is augmented if $\epsilon\sigma = \epsilon$  and $\epsilon\rho = 0$.

Note that $\eqref{eq:augwreath}$ can hold for more general $\sigma$ and $\rho$. For examples, those satisfying $\sigma(b_1 \otimes b_2) = k(b_2\otimes b_1)$ for some unit $k\in K^{\times}$, and
$h(1-k) \epsilon = \epsilon \rho \in \Hom(B\otimes B, K)$.
}
Under the assumption \eqref{eq:augwreath}, we can provide more information about the quantum wreath products and their quotients.
\begin{prop}\label{prop:struc}
Let $A = B \wr \cH(d)$ be a quantum wreath product with two bases as in Theorem~\ref{thm:basis}. 
\begin{itemize}
\item[(a)]
Assume that {both $A$ and $B$ are augmented algebras with counits both denoted by $\epsilon$}.
If $\sigma(b_1 \otimes b_2) \in Kb_2\otimes b_1$, 
then
$B^{\otimes d}$ is a normal subalgebra of $A$.	

As a consequence,
	$A // B^{\otimes d}$ is isomorphic to the {associative $K$-algebra generated by $T_1, \dots, T_{d-1}$ subject to the usual braid relations of type A together with modified quadratic relations $T_i^2 = \epsilon(S)T_i + \epsilon(R)$. Such an algebra is denoted by $\cH_{\epsilon}(\Sigma_d)$ since it recovers the Hecke algebra $\cH_q(\Sigma_d)$ when $\epsilon(S) = (q-1), \epsilon(R) =q$ for some $q\in K^\times$.}
\item[(b)]
Assume that $B$ is a symmetric algebra with trace map $\Tr: B\to K$ that induces to $\Tr: B^{\otimes d}\to K$, $b_1\otimes \cdots \otimes b_d \mapsto \prod_i\Tr(b_i)$.
If $\Tr (\sigma(b)) = \Tr(b)$ for  $b \in B\otimes B$, $R$ is invertible, and $\rho = 0$, then $A$ is symmetric.
\end{itemize}
\end{prop}
\begin{proof}
(a) Since $(B^{\otimes d})^+ \lhd B^{\otimes d}$ is an ideal, an arbitrary element in 
$A(B^{\otimes d})^+ 
= \langle H_i \rangle (B^{\otimes d})^+$ 
is of the form $\sum_{w}H_wb_w$, 
where $\epsilon(b_w) =0$ for all $w$. 
Thus, it suffices to show that for any $b \in B^{\otimes d}$ with $\epsilon(b) =0$,  
$H_ib \in (B^{\otimes d})^+A 
= (B^{\otimes d})^+ \langle H_j\rangle$ for all $i$.

Since $\sigma(b_1 \otimes b_2) \in Kb_2\otimes b_1$, $\epsilon(\sigma_i(b))$ is a multiple of $\epsilon(b)=0$. Thus,
\eq
H_ib = \sigma_i(b)H_i + \rho_i(b) \in (B^{\otimes d})^+ \langle H_j\rangle = (B^{\otimes d})^+A,
\endeq
where $\rho_i(b) \in (B^{\otimes d})^+$ since {$\epsilon(\rho(b)) \in K  \epsilon(b) =0$ for all $b\in (B^{\otimes d})^+$, thanks to \eqref{eq:augwreath}}.


(b) 
Since $B$ is symmetric, $\beta: B\times B \to K, (a,b) \mapsto \Tr(ab)$ is a non-degenerate associative symmetric bilinear form. 
We claim that such a form for $A$ is $\beta_A:A \times A \to K$, given by
\eq
\beta_A(aH_x, bH_y) = \Tr(aH_x bH_y),
\quad
\textup{where}
\quad
\Tr(bH_w) :=\begin{cases}
\Tr(b) &\tif w=1;
\\
0 &\textup{otherwise}.
\end{cases}
\endeq  
This trace map is well-defined since $A$ has the bases as in  Theorem~\ref{thm:basis}.
It is not hard to verify that $\beta_A$ is nondegenerate and associative {(for this to be true, it is crucial  that $R$ is invertible)}.
It suffices to check that $\Tr(aH_x bH_y) = \Tr(bH_y aH_x )$ for all $a, b\in B^{\otimes d}, x,y \in \Sigma_d$.
By applying the wreath relation iteratively, we obtain $H_ya = \sigma_{y}(a) H_y$, 
where $\sigma_{y}:= \sigma_{i_1} \cdots \sigma_{i_N}$ for a reduced expression $y = s_{i_1} \cdots s_{i_N} \in \Sigma_d$,
which is well-defined as a consequence of \eqref{def:br1}.
Hence, for all $a, b \in B^{\otimes d}$, 
\eq
\Tr(bH_yaH_x) =\Tr(b\sigma_y(a) H_yH_x),
\quad
\Tr(aH_x b H_y) =\Tr(a\sigma_x(b) H_xH_y).
\endeq
Both traces are zero unless $x=y\inv$.
In the case $x = y\inv$, 
assume that $x = s_{i_1} \cdots s_{i_N}$ is a reduced expression. 
If we write $H_y H_x = \sum_{z\in \Sigma_d} c_z H_z$, then
\eq
c_1 = (\sigma_{i_{N}}\cdots \sigma_{i_{2}})(R_{i_{1}})(\sigma_{i_{N}}\cdots \sigma_{i_{3}})(R_{i_{2}})\cdots \sigma_{i_{N}}(R_{i_{N-1}})R_{i_N}.
\endeq
On the other hand, if we write $H_x H_y = \sum_{z\in \Sigma_d} c'_z H_z$, then
\eq
c'_1 = (\sigma_{i_{1}}\cdots \sigma_{i_{N-1}})(R_{i_{N}})\cdots \sigma_{i_{1}}(R_{i_{2}})R_{i_1}.
\endeq
Note that when $\rho = 0$ and $R$ is invertible, \eqref{def:TTT} become $\sigma(S) = S$ and $\sigma(R) = R$.
Since $\sigma_x =  \sigma_{i_1} \cdots \sigma_{i_N}$ preserves the trace, 
\eq
\begin{split}
\Tr(b\sigma_y(a)H_x H_y) 
&= 
\Tr(b\sigma_y(a) (\sigma_{i_{N}}\cdots \sigma_{i_{2}})(R_{i_{1}})\cdots \sigma_{i_{N}}(R_{i_{N-1}})R_{i_N})  
\\
&= \Tr(\sigma_x(b)a \sigma_{i_1}(R_{i_1}) (\sigma_{i_{1}}\sigma_{i_{2}})(R_{i_{2}})\cdots (\sigma_{i_{1}}\cdots \sigma_{i_{N}})(R_{i_{N}}))
\\
&= \Tr(\sigma_x(b)a R_{i_1} \sigma_{i_{1}}(R_{i_{2}})\cdots (\sigma_{i_{1}}\cdots \sigma_{i_{N-1}})(R_{i_{N}}))
=  \Tr(a\sigma_x(b) H_x H_y), 
\end{split}
\endeq
and so the two traces agree, i.e., $A$ is a symmetric algebra.
\end{proof}
\section{Quantum Wreath Product Algebras} \label{sec:examples}
{In the following, a list of algebras is provided which are special cases of our quantum wreath products. 
For each case, the reader is referred to Appendix~\ref{sec:app} for a presentation of the algebra $A$ in question by generators and relations.
We choose suitably our parameters to realize these algebras as a quantum wreath product $B \wr \cH(d)$.
In order to show that $A \simeq B \wr \cH(d)$, one can either show that the relations for $A$ are equivalent to those for $B\wr \cH(d)$, under a certain identification;
or use a dimension argument when $A$ is finite dimensional.}
\subsection{Hu Algebras}\label{ex:HHS}
The reader is referred to Section~\ref{sec:Hu} for the definition of the Hu algebra $ \cA(m)$. The algebra $\cA(m)$ is essentially different from all other examples in Section~\ref{sec:examples}.
It is not covered by the theory of Rosso-Savage nor of Kleshchev-Muth due to its intriguing quadratic relation.

Assume that $K$ is a field of characteristic not equal to 2, $f_{2m}(q)\neq 0$,  $B=\cH_q(\Sigma_m)$, and the parameter $Q = (R,S,\rho,\sigma)$ is given by 
\eq
R = z_{m,m},
\quad
S = 0, 
\quad 
\sigma: a\otimes b \mapsto b\otimes a,
\quad
\rho = 0.
\endeq
For this choice, almost all conditions \eqref{def:wr1}--\eqref{def:br5}  follow immediately (see \eqref{eq:simQ}).
The only non-trivial condition, $\sigma(R) = R$, follows the construction of $z_{m,m}$ in Section~\ref{sec:Am}.

The algebra $\cH_q(\Sigma_m) \wr \cH(2)$ is generated by 
$T_x \otimes T_y\  (x, y \in \Sigma_m)$ and $H_1$.
Note that there are no braid relations for a single generator $H_1$. The quadratic relation translates to that $H_1^2 = z_{m,m}$, and the wreath relation becomes
\eq
H_1(T_x\otimes T_y) = (T_y \otimes T_x)H_1 \quad (x, y \in \Sigma_m).
\endeq
A dimension argument (by Proposition~\ref{prop:Hu}) shows that the Hu algebra $ \cA(m)$ is canonically isomorphic to $\cH_q(\Sigma_m) \wr \cH(2)$ via the assignment below:
\eq
T_i \mapsto \begin{cases}
T_i \otimes 1 &\tif 1\leq i \leq m-1;
\\
1\otimes T_{i-m} &\tif m+1 \leq i \leq 2m-1,
\end{cases}
\quad
H_1(m) \mapsto H_1.
\endeq

\subsection{(Degenerate) Affine Hecke Algebras}\label{sec:AHAfamily}
See Appendix~\ref{sec:AHA} for definitions of variants of affine Hecke algebras.
Now we consider the case when the base algebra $B$ is generated by $X$ {and that $\sigma: a\otimes b \mapsto b\otimes a$ is the flip map}.
Verifying relations for all these variants can be made easier via the Demazure operator  $\partial$, given by
\eq\label{eq:partial}
\partial(X^a \otimes X^b) = \frac{X^a\otimes X^b - X^b \otimes X^a}{X\otimes 1 - 1\otimes X}.
\endeq
In particular, $\partial(X\otimes 1) = 1 = - \partial(1\otimes X)$.
\begin{prop}\label{prop:De}
{Let $X \in B$ be an element not in the base ring $K$}. Let $\sigma  \in \End(B\otimes B)$ be the flip map, and let $\rho \in \End(B\otimes B)$ be such that \eqref{def:wr2} holds.
If  $\beta := \rho(X\otimes 1)$ commutes with both $X\otimes 1$ and $1\otimes X$,
then the following holds for all $i,j \geq 0$: 
\eq\label{eq:Derho}
\rho(X^i \otimes X^j) = \partial(X^i\otimes X^j)\beta.
\endeq
Moreover, if $X\inv \in B$, then \eqref{eq:Derho} also holds for $i,j \in \ZZ$.
\end{prop}
\proof
First, it follows from \eqref{def:wr2} that
\eq\label{eq:X11X}
\rho(X \otimes X) = \rho((X\otimes 1)(1\otimes X)) = (1\otimes X) (\rho(1\otimes X)+ \beta). 
\endeq
On the other hand, $\rho(X \otimes X)$ is also equal to
$\rho((1\otimes X)(X\otimes 1)) = (X\otimes 1) (\beta + \rho(1\otimes X))$, and hence
 $\rho(1 \otimes X) = -\beta$. An inductive argument shows that, for all $i \geq 1$,
\eq
\rho(X^i\otimes 1) = \beta \partial(X^i\otimes 1) = -\rho(1\otimes X^i),
\endeq
and hence,  
\eq
\begin{split}
\rho(X^i \otimes X^j) &= (1\otimes X^i) \rho(1\otimes X^j) + \rho(X^i \otimes 1)(1\otimes X^j) 
\\
&=\beta ((X^{i-1}\otimes X^j + \dots + 1\otimes X^{i+j-1}) - (X^{j-1}\otimes X^i + \dots + 1\otimes X^{i+j-1}))
\\
&= \beta \partial(X^i \otimes X^j).
\end{split}
\endeq
This first part is done. Assume from now on {$X\inv \in B$}. 
We apply \eqref{def:wr2} to both $\rho((X\otimes 1)(1\otimes X\inv))$ and $\rho((1\otimes X\inv)(X\otimes 1))$ and then equate the two. We then obtain \eqref{eq:Derho} for ${i=0, j<0}$. Similarly, by applying \eqref{def:wr2} to both $\rho((X\inv\otimes 1)(1\otimes X))$ and $\rho((1\otimes X)(X\inv\otimes 1))$ and then equating the two, we obtain \eqref{eq:Derho} for ${i<0, j=0}$. The most general case is then obtained by applying \eqref{def:wr2} to $\rho((X^i \otimes 1)(1\otimes X^j))$.
\endproof
In other words, for {the} verification of  {the} conditions appearing in Theorem~\ref{thm:basis} involving $\rho$, one only need to check them on the generator $X\otimes 1$ as long as  $\beta := \rho(X\otimes 1)$ commutes with both $X\otimes 1$ and $1\otimes X$.

Now, for any $\mu = \sum_{i=1}^d \mu_i \epsilon_i \in \sum_i \ZZ \epsilon_i$, write $X^\mu \equiv X^{\mu_1}\otimes \dots \otimes X^{\mu_d}$, and let $s_j$ act on $\sum_i \ZZ \epsilon_i$ by place permutation on the $j$th and $(j+1)$th positions.
{For an element $X\in B$, for $1\leq j \leq d$, set 
\eq
X^{(j)}
:= \underset{j-1 \textup{ factors}}{\underbrace{1 \otimes \dots \otimes 1}} \otimes X \otimes  \underset{d-j \textup{ factors}}{\underbrace{1 \otimes \dots \otimes 1}} \in B^{\otimes d}.
\endeq}
\begin{cor}[Bernstein-Lusztig Relations]\label{cor:BL}
Retain the assumptions in Proposition~\ref{prop:De}.
The following {local equalities in $B \wr \cH(d)$ hold in the sense of \eqref{eq:localrel}}:
\eq\label{eq:HXH}
H (X\otimes 1) H= (1\otimes X)R+ ((1\otimes X)S + \beta) H.
\endeq
\eq\label{eq:BLrho}
H (X^i\otimes X^j) = (X^j\otimes X^i)H + \partial(X^i\otimes X^j)\beta.
\endeq
{In particular, the following relations in $B \wr \cH(d)$ hold, for $1\leq i \leq d-1$, and for all $\mu$}, 
\eq\label{eq:BLrho2}
H_i X^\mu = X^{s_i(\mu)}H_i + \frac{X^\mu - X^{s_i(\mu)}}{X^{\epsilon_{i}}-X^{\epsilon_{i+1}}}\rho_i(X^{(i)}).
\endeq
\end{cor}
\proof 
Equations \eqref{eq:HXH} and \eqref{eq:BLrho} follow from direct computations {using \eqref{def:QR2}--\eqref{def:WR2}};
while \eqref{eq:BLrho2} is a paraphrase of \eqref{eq:BLrho}.
\endproof
\subsubsection{Affine Hecke Algebras of Type A}\label{ex:AHA}
Let $B = K[X^{\pm1}]$ be the Laurent polynomials in $X$ over a field $K$, $q \in K^{\times}$, and let
\eq\label{eq:RRAff}
\begin{split}
&R = q(1\otimes1),
\quad
\sigma: {f\otimes g \mapsto g \otimes f, \quad (f,g\in B)},
\\
&S= (q-1)(1\otimes1),
\quad
\rho: X^a \otimes X^b \mapsto -(q-1)\partial(X^a \otimes X^b)(1\otimes X).
\end{split}
\endeq
The extended affine Hecke algebra $\cH^\ext_q(\Sigma_d)$ 
can be realized via the quantum wreath product
$K[X^{\pm1}] \wr \cH(d)$ under the identification, for $\lambda = (\lambda_1, \dots, \lambda_d)$:
\eq
T_i \mapsto H_i,
\quad
Y^\lambda \mapsto X^\lambda \equiv X^{\lambda_1} \otimes \dots \otimes X^{\lambda_d}.
\endeq
Since $\beta = -(q-1) 1\otimes X$ commutes with either $X\otimes 1$ or $1\otimes X$,  Corollary~\ref{cor:BL} follows.
Therefore, the wreath relations \eqref{def:WR2} implies the Bernstein-Lusztig relations \eqref{eq:BLrel}
since  
$$\frac{X^\mu - X^{s_i(\mu)}}{X^{\epsilon_{i}}-X^{\epsilon_{i+1}}}\beta 
=
(q-1) \frac{X^\mu - X^{s_i(\mu)}}{1-X^{\epsilon_i - \epsilon_{i+1}}}.$$
It is easy to see that the Bernstein-Lusztig relation implies the wreath relation, and hence, they are equivalent.

{
Since it is well-known (see \cite{Lu89}) that $\cH^\ext_q(\Sigma_d)$ admits a basis of the form $\{Y^\mu T_w ~|~ \mu \in P, w\in\Sigma_d\}$, conditions \eqref{def:wr1}--\eqref{def:br5} hold as long as Theorem~\ref{thm:basis} is proved.
On the other hand, one can give another proof of Lusztig's basis theorem by a tedious but elementary verification of conditions \eqref{def:wr1}--\eqref{def:br5}.
}

Note that by this choice of $R, S$, we have $H(X\otimes 1)H = q(1\otimes X)$ from Corollary~\ref{cor:BL}, or equivalently, $H_i X^{(i)} H_i = qX^{(i+1)}$ for all $i$.
If we choose $R = 1\otimes 1, S = (q\inv - q)1\otimes 1$ instead, we would obtain $H_i X^{(i)} H_i = X^{(i+1)}$ for all $i$.
 
The affine Hecke algebra $\cH_q(\Sigma_d^\aff)$ is then realized as the subalgebra
\eq
\cH_q(\Sigma_d^\aff) \equiv \{X^\lambda H_w ~|~ \Sigma_i \lambda_i = 0, w\in \Sigma_d\}.
\endeq
In particular, the affine generator $T_0\in \cH_q(\Sigma_d^\aff)$ can be identified as
\eq
T_0 \equiv (X \otimes 1^{\otimes d-2} \otimes X\inv) H_{d-1} \dots H_2 H_1 H_2 \dots H_{d-1}.
\endeq 
\subsubsection{Ariki-Koike Algebras} \label{ex:AK}
Let $q_1, \dots, q_m\in K$, and let $Q= (R, S, \rho, \sigma)$ be as in \eqref{eq:RRAff}.
The Ariki-Koike algebra is the following cyclotomic quotient of the quantum wreath product
\eq 
\cH_{q,q_1, \dots, q_m}(C_m\wr \Sigma_d) = \frac{K[X^{\pm1}] \wr \cH(d)}{\langle (X-q_1) \dots (X-q_m)\otimes 1^{\otimes d-1}\rangle}.
\endeq
In particular, the Jucys-Murphy elements $L_i$'s are identified as below:
\eq
L_1 = T_0 \equiv \={X \otimes 1^{\otimes d-1}} =: \={X^{(1)}},
\quad
L_{i+1} = {q\inv} T_i L_i T_i \equiv \={X^{(i+1)}} 
\quad
\textup{for}
\quad 0 \leq i \leq d-1.
\endeq
Therefore, the well-known surjection  $\cH_q^\ext(\Sigma_d) \to \cH_{q,q_1, \dots, q_m}(C_m\wr \Sigma_d), Y^{\epsilon_i} \mapsto L_i$ is the canonical map
\eq
K[X^{\pm1}] \wr \cH(d) \to  \frac{K[X^{\pm1}] \wr \cH(d)}{\langle (X-q_1) \dots (X-q_m)\otimes 1^{\otimes d-1}\rangle},
\quad
X^\mu H_w \mapsto \={X^\mu H_w} 
\quad
(\mu \in \ZZ^d, w\in \Sigma_d).
\endeq
\subsubsection{Affine Hecke Algebras of Type B/C}
\label{ex:AHABC}

The (non-extended) affine Hecke algebra $\cH(C^\aff_d)$ of type C is the Hecke algebra $\cH(C^\aff_d) = \<T_0, \dots, T_d\>$ for the Coxeter group of type $C^\aff_d$, in which $T_0$ corresponds to the type C node in the Dynkin diagram; while $T_d$ corresponds to the affine node.
{Let $B = K[X^{\pm1}, Y^{\pm1}]$, $q, \xi, \eta \in K^{\times}$, and let
\eq\label{eq:RRAffC}
\begin{split}
&R = (1\otimes1),
\quad
\sigma:  f\otimes g \mapsto g \otimes f, \quad (f,g\in B),
\\
&S= (q\inv-q)(1\otimes1),
\quad
\rho: 
\begin{cases}
X \otimes 1 \mapsto (q-q\inv)(1\otimes X);
\\
Y \otimes 1 \mapsto (q\inv-q)(Y\otimes 1).
\end{cases}
\end{split}
\endeq
}
Then 
\eq
\cH(C^\aff_d) =
\frac{K[X^{\pm1}, Y^{\pm1}] \wr \cH(d)}{\langle (X+\xi)(X-\xi\inv)\otimes 1^{\otimes d-1}, 1^{\otimes d-1} \otimes (Y+\eta)(Y-\eta\inv)\rangle},
\endeq
under the identification
\eq
T_0 \mapsto \overline{X^{(1)}},
\quad
T_i \mapsto \overline{H_i} \quad (1\leq i \leq d-1),
\quad
T_d \mapsto \overline{Y^{(d)}}.
\endeq
On the other hand, the extended affine Hecke algebra  $\cH^\ext(B_d)$ of type B is generated by
$T_0$, $T_1$, $\dots$, $T_{d-1}$, $Y^\lambda (\lambda \in P(B^\vee_d))$ where $T_0$ corresponds to the type B node in the Dynkin diagram. 
It is possible to identify $\cH^\ext(B_d)$ as a subalgebra of $K[X^{\pm1}] \wr \cH(2d)$ (with respect to \eqref{eq:RRAff}) generated by
\eq
\begin{split}
&X^{\lambda_1} \otimes \dots \otimes X^{\lambda_{2d}} \quad ( \lambda_i = - \lambda_{2d+1-i} \textup{ for all }i), 
\\
&H_d, \quad H_{d+1}H_{d-1},\quad \dots,\quad H_1H_{2d-1}.
\end{split} 
\endeq
\subsubsection{Degenerate Affine Hecke Algebras}
\label{ex:dAHA}
Let $\cH^\Deg(d)$ be the degenerate affine Hecke algebra of type A, generated by $s_1, \dots, s_{d-1}, x_i, \dots, x_d$ (see Appendix~\ref{sec:AHA}). The $s_i$'s and $x_j$'s interact by the cross relations \eqref{eq:sxrel}.
Let $B = K[X]$ be the polynomial ring.
We choose
\eq
R = 1\otimes 1, \quad S = 0, \quad \sigma: X^a\otimes X^b \mapsto X^b\otimes X^a, \quad 
\rho: X^a \otimes X^b
\mapsto -\partial(X^a \otimes X^b).
\endeq
Then $\cH^\Deg(d) = K[X] \wr \cH(d)$ under the identification
\eq
s_i \mapsto H_i, 
\quad x_j \mapsto X^{(j)}.
\endeq
For $\lambda = \sum_i \lambda_i \epsilon \in P^+ := \sum_{i=1}^d \ZZ_+\epsilon_i$, we set $x^\lambda := x_1^{\lambda_1} \dots x_d^{\lambda_d}$.
Now, $\beta = - 1$ commutes with either $X\otimes 1$ or $1\otimes X$, and hence Corollary~\ref{cor:BL} follows.
Similar to the case for affine Hecke algebras,
our wreath relations \eqref{def:WR2} are equivalent to the Bernstein-Lusztig-type relations below, for $\lambda \in P^+, 1\leq i \leq d-1$: 
\eq\label{eq:dBLrel}
s_i x^\lambda= x^{s_i(\lambda)} s_i - \frac{x^\lambda - x^{s_i(\lambda)}}{x_i - x_{i+1}}.
\endeq
{
Since it is well-known (see \cite{Lu89}) that $\cH^\Deg(d)$  admits a basis of the form $\{x^\lambda w ~|~ \lambda \in P^+, w\in\Sigma_d\}$, conditions \eqref{def:wr1} -- \eqref{def:br5} hold as long as Theorem~\ref{thm:basis} is proved.
}

Kostant-Kumar's Nil Hecke rings (see Appendix~\ref{sec:NHR}) can be realized similarly by setting $R=0$, instead.
{
\subsubsection{Hecke-Type Categories}\label{sec:HC}
It is known that the Nil Hecke rings describe the endomorphism spaces in the Khovanov-Lauda-Rouquier (KLR) algebras. This generalizes to 
Elias' Hecke-type category, whose endomorphism spaces (of rank $d$) can be regarded as an algebra generated by 
$H_1, \dots, H_{d-1}$ and $B^{\otimes d}$ (where $B$ is a commutative ring; in contrast to our definition in which $B$ is an arbitrary associative algebra).
Their algebras satisfy our quadratic relations \eqref{def:QR2} as well as wreath relations \eqref{def:WR2}; while the braid relations are relaxed as follows:
\eq
H_{i} H_{i+1} H_i \in K^{\times} H_{i+1} H_iH_{i+1} +\sum_{x < s_i s_{i+1} s_i} (B\otimes B\otimes B) H_x.
\endeq
The key point is that our quantum wreath products is a special case of Elias' construction when the base algebra $B$ is commutative.
However, in our main example, the Hu algebra, the base algebra $B= \cH_q(\Sigma_m)$ is far from commutative.
}
\subsection{Rosso-Savage Algebras}
\label{ex:RS}
Let $z\in K$, and let $B$ be a certain Frobenius $K$-algebra with a Nakayama automorphism $\psi:B\to B$,  a trace map $\textup{tr}:B\to K$, a basis $I$ and its dual basis $\{b^\vee~|~b\in I\}$ with respect to $\textup{tr}$. Rosso and Savage introduced in \cite{Sa20, RS20} the Frobenius Hecke algebra $H_d(B,z)$, its affinization $H^\aff_d(B,z)$, and the affine wreath product algebra $\cA_d(B)$ which generalize the Hecke algebra, the affine Hecke algebra, and the degenerate affine Hecke algebra, respectively (see Appendix~\ref{def:RS} for details).

We have the following identifications as quantum wreath products:
\eq
H_d(B,z) = B\wr \cH(d),
\quad
H^\aff_d(B,z) = (B\hat{\otimes} K[X^{\pm1}]) \wr \cH(d),
\quad
\cA_d(B) = (B\hat{\otimes} K[X]) \wr \cH(d),
\endeq
where  $\hat{\otimes}$ denotes the twisted tensor product  given by ${(b\otimes1)X^{\pm1} = X^{\pm1}(\psi^{\pm1}(b) \otimes 1)}$ for $b\in B$.
Here $R = 1\otimes 1$ and $\sigma(a\otimes b) = b\otimes a$ for all three cases.
For $H_d(B,z)$, the other parameters are
\eq
S=z \sum_{b \in I} (b\otimes b^\vee) \in B\otimes B, 
\quad
\rho = 0.
\endeq
For $H^\aff_d(B,z)$, one has 
\eq
S=z\sum_{b \in I} (b\otimes b^\vee) \in B\otimes B, 
\quad
\rho: \begin{cases}
b_1\otimes b_2 \mapsto 0 &(b_i \in B);
\\
X \otimes 1 \mapsto -(1\otimes X) z \sum_{b \in I} (b\otimes b^\vee) \in B\otimes B. 
\end{cases}
\endeq
For $\cA_d(B)$, one has 
\eq
S=0, 
\quad
\rho: \begin{cases}
b_1\otimes b_2 \mapsto 0 &( b_i \in B);
\\
X \otimes 1 \mapsto - \sum_{b \in I} (b\otimes b^\vee) \in B\otimes B. 
\end{cases}
\endeq
In particular, these algebras include the Yokonuma-Hecke algebra $Y_{m,d} = KC_m \wr \cH(d)$, its affinization, 
Wan-Wang's \cite{WW08} wreath Hecke algebra $\cH^\Deg(G\wr \Sigma_d) = K[X]G \wr \cH(d)$,
Evseev-Kleshchev's super wreath product algebra.

{In \cite{SS21}, Savage and Stuart also considered variants of these algebra by replacing the quadratic relation with $H_i^2=0$. 
By a similar argument, one can show that their Frobenius nilCoxeter algebras are also special cases of ours.
Their Frobenius nilHecke algebras can also be included as long as the base algebra is purely even.}
\subsection{Affine Zigzag Algebras}
\label{ex:KM}
Readers are referred to Section~\ref{def:KM} for the definition of the zigzag algebra $Z_K(\Gamma)$, its affinization $Z^\aff_d(\Gamma)$ over a connected Dynkin diagram $\Gamma$ of finite type ADE.
Let $B = Z_{K[X]}(\Gamma)$ be the zigzag algebra over the polynomial ring $K[X]$.
Let $R = 1\otimes 1,  S = 0,  \sigma: a\otimes b \mapsto b\otimes a$, and let $\rho: B\otimes B \to B \otimes B$ be determined by, for $\gamma_i \in Z_K, i \neq j \in I$: 
\eq
\rho(\gamma_1 \otimes \gamma_2) = 0,
\quad
\rho(Xe_i \otimes e_i) = (c_1+c_2) (e_i \otimes e_i),
\quad
\rho(Xe_i \otimes e_j) = a_{j,i} \otimes a_{i,j}.
\endeq
In other words, we have the following local relations, for $\gamma_i \in Z_K, i \neq j \in I$:
\eq
\begin{split}
&H(\gamma_1\otimes \gamma_2) = (\gamma_2 \otimes \gamma_1) H,
\\
&H(Xe_i \otimes e_i) = (e_i \otimes Xe_i)H + (c_1+c_2) (e_i \otimes e_i),
\quad
H(Xe_i \otimes e_j) = (e_j \otimes Xe_i)H + a_{j,i} \otimes a_{i,j}.
\end{split}
\endeq
Then $Z_d^\aff(\Gamma) = Z_{K[X]} \wr \cH(d)$ under the identification
with $T_i \mapsto H_i$, $z_j \mapsto X^{(j)}$.
\section{Basis Theorem for Quantum Wreath Products}\label{sec:RepA}
This section is dedicated to determining the necessary and sufficient conditions on the choice of parameters such that that our quantum wreath product produces a unital associative algebra having a basis as given in 
Theorem~\ref{thm:basis}.

\subsection{}
{We start with a left $B^{\otimes d}$-module $V :=  B^{\otimes d} \otimes K \Sigma_d$, on which $B^{\otimes d}$ acts by left multiplication. Let $A := B\wr \cH(d)$. Our goal is to show that, under suitable conditions, $V$ can be given a right $A$-module structure, and the assignment $a\otimes w \mapsto a H_w$ for $a\in B^{\otimes d}, w\in \Sigma_d$ defines a right $A$-module isomorphism $V \to A$.
Note that $V$ admits a filtration $(V^\ell)_{\ell \in \NN}$ of $B^{\otimes d}$-modules given by
\eq
	V^\ell := \bigoplus_{\ell(w) \leq \ell} B^{\otimes d} \otimes K w.
\endeq

We fix a reduced expression $r(w)$ for all $w\in \Sigma_d$ that is compatible with others, i.e., $r(1) = 1$ and for every $w \in \Sigma_d \setminus \{1\}$, there exists a unique $s_i$ such that $r(w) = r(w s_i) s_i$ as reduced expressions.
Existence of such a choice follows from an inductive argument.

In order to eliminate the left-versus-right issue, in Section 5, we use the following postfix notation such that the juxtaposition of maps means the reversed composition of maps, i.e., for any two maps $f:M_1 \to M_2, g:M_2 \to M_3$, it is understood that
\eq\label{eq:postfix}
fg := g \circ f : M_1 \to M_3,
\quad
\textup{and}
\quad
(m)\cdot fg := (g\circ f)(m) \in M_3
\quad
\textup{for all}
\quad
m\in M_1,
\endeq
where $m\cdot f = f(m)$ is the postfix evaluation of $f$ at $m$.
In particular, for endomorphisms $f, g \in \End(M)$, by the juxtaposition $fg$ we mean the the multiplication in the opposite ring $\End(M_1)^{\textup{op}}$.


Next,  we define  (left) $B^{\otimes d}$-module homomorphisms
$f^{(\ell)}_a = f^{(\ell)}_a(r) \in \End(V^\ell)$, and $T_i^{(\ell)} = T_i^{(\ell)}(r):V^\ell \to V^{\ell+1}$, for all $a\in B^{\otimes d}, 1\leq i \leq d-1, \ell \geq 0$ in a recursive manner:
\eq
\label{def:alaction}
(1^{\otimes d}\otimes w) \cdot f_{a}^{(\ell)}
	:= \begin{cases}
	a \otimes 1, & \tif w = 1, \\
		(1^{\otimes d} \otimes w s_i) \cdot ( f_{\sigma_i(a)}^{(\ell - 1)} T_{i}^{(\ell - 1)} + f_{\rho_i(a)}^{(\ell - 1)}), & \tif  r(w) = r(w s_i) s_i, \\
	\end{cases}
\endeq
\eq
\label{def:Tilaction}
(1^{\otimes d} \otimes w) \cdot T_i^{(\ell)} := \begin{cases}
1^{\otimes d}\otimes ws_i , &\tif \ell(w) \lt \ell(w s_i), 
	\\
(1^{\otimes d}\otimes w s_i) \cdot (f_{S_i}^{(\ell-1)} T_i^{(\ell-1)} + f_{R_i}^{(\ell-1)}), &\tif \ell(w) \gt \ell(w s_i),
	\end{cases}
\endeq
for all $w$ with length $\ell(w) \leq \ell$.
Note that the recursion does terminate since $T_i^{(0)}$ and $f_a^{(0)}$ are given by $(b \otimes 1) \cdot T_i^{(0)} = b\otimes s_i$ and $(b\otimes 1) \cdot f_{a}^{(0)} = (ba)\otimes 1$ for any $b\in B^{\otimes d}$.
}
\begin{prop}\label{prop:fTcomp}
For a fixed $a \in B^{\otimes d}$ (or a fixed $1\leq i\leq d-1$, resp.), the maps $\{f_a^{(\ell)} ~|~ \ell \geq 0\}$ (or $\{T_i^{(\ell)} ~|~ \ell \geq 0\}$, resp.) are compatible, i.e.,
$f_a^{(\ell)}|_{V^j} = f_a^{(j)}$ and $T_i^{(\ell)}|_{V^j} = T_i^{(j)}$ for all $j \leq \ell$.
\end{prop}

\proof
We prove this by an induction on $j$.
The base case $f_a^{(\ell)}|_{V^0} = f_a^{(0)}$ and $T_i^{(\ell)}|_{V^0} = T_i^{(0)}$ follow from \eqref{def:alaction} and \eqref{def:Tilaction}, respectively. 
%
For the inductive case, pick any $b\otimes x \in V^j$ with $x \neq 1$ (i.e., $\ell(x) \leq j \leq \ell$). There is an $s_k$ such that $r(x) = r(w)s_k$, and thus $b\otimes w \in V^{j-1}$.
Then
\eq\label{eq:fal}
\begin{array}{ll}
(b\otimes x) \cdot f_a^{(\ell)} = (b\otimes w) \cdot (f_{\sigma_k(a)}^{(\ell-1)}T_k^{(\ell-1)}  + f_{\rho_k(a)}^{(\ell-1)})
&\textup{by }\eqref{def:alaction}
\\
\quad =(b\otimes w) \cdot (f_{\sigma_k(a)}^{(j-1)}T_k^{(j-1)}  + f_{\rho_k(a)}^{(j-1)})
&\textup{by the induction hypothesis}
\\
\quad = (b\otimes x) \cdot f_a^{(j)}. 
&\textup{by }\eqref{def:alaction}
\end{array}
\endeq
The proof of that $T_i^{(\ell)}|_{V^j} = T_i^{(j)}$ for all $j \leq \ell, 1\leq i \leq d-1$ splits into two cases.
The first case $\ell(x) > \ell(xs_i)$ follows from the fact that
\eq
\begin{array}{ll}
(b\otimes x)\cdot T_i^{(\ell)} = (b\otimes xs_i)\cdot
( f_{S_i}^{(\ell-1)}T_i^{(\ell-1)} + f_{R_i}^{(\ell-1)})
&\textup{by }\eqref{def:Tilaction}
\\
\quad = (b\otimes xs_i) \cdot (f_{S_i}^{(j-1)}T_i^{(j-1)} + f_{R_i}^{(j-1)}) 
&\textup{by the induction hypothesis}
\\
\quad = (b\otimes x) \cdot T_i^{(j)}.  
&\textup{by }\eqref{def:Tilaction}
\end{array}
\endeq
The other case $\ell(x) < \ell(xs_i)$ follows from the fact that $(b\otimes x)\cdot T_i^{(\ell)} = b\otimes xs_i = (b\otimes x) \cdot T_i^{(j)}$, thanks to \eqref{def:Tilaction}.
\endproof

Let $f_a = f_a(r), T_i = T_i(r) \in \End_{B^{\otimes d}}(V)$ be given by $f_a|_{V^\ell} = f_a^{(\ell)}$ and $T_i|_{V^\ell} = T_i^{(\ell)}$ for all $\ell$.  Thanks to Proposition~\ref{prop:fTcomp}, these maps are well-defined, and it leads to the following proposition from \eqref{def:alaction}--\eqref{def:Tilaction}.

\begin{prop} \label{prop:ATonV}
The following statements hold for all $w\in \Sigma_d, a\in B^{\otimes d}, 1\leq i \leq d-1$:
\begin{align}
\label{def:aaction}&	(1^{\otimes d}\otimes w) \cdot f_{a}
	:= \begin{cases}
	a \otimes 1, & \tif w = 1, \\
		(1^{\otimes d} \otimes w s_i) \cdot ( f_{\sigma_i(a)} T_{i} + f_{\rho_i(a)}), & \tif  r(w) = r(w s_i) s_i. \\
	\end{cases}
\\
\label{def:Tiaction}&	 ( 1^{\otimes d}\otimes w) \cdot T_i = \begin{cases}
	 1^{\otimes d}\otimes ws_i  , &\tif \ell(w) \lt  \ell(ws_i), \\
	(1^{\otimes d}\otimes w s_i) \cdot (f_{S_i} T_i + f_{R_i}), &\tif \ell(w) \gt \ell(ws_i). \\
\end{cases}
\end{align}
\end{prop}

We remark that one cannot define $f_a(r)$ and $T_i(r)$ directly using \eqref{def:aaction}--\eqref{def:Tiaction} since a circular definition will occur. 
Later, we will show that these maps are actually independent of $r$. 

{
\begin{cor}\label{prop:Klinearity}
The map $f: B^{\otimes d} \to \End_{B^{\otimes d}}(V)^{\textup{op}}, a \mapsto f_a$ is $K$-linear.
In particular, $f_a+f_b = f_{a+b}$ for all $a, b\in B^{\otimes d}$.
\end{cor}

\begin{proof}
It suffices to prove that $f$ is $K$-linear on $1^{\otimes d} \otimes w$ for any $w \in \Sigma_d$.
From \eqref{def:aaction}, it is clear that $f$ is $K$-linear in $1^{\otimes d} \otimes w$ if $w=1$.
If $\ell(w) > 1$, then there is an $s_i$ such that $r(w) = r(ws_i)s_i$.
By an induction on $\ell(w)$, $f$ is $K$-linear on  $1^{\otimes d} \otimes w$ by combining \eqref{def:aaction} and the fact that both $\sigma_i$ and $\rho_i$ are $K$-linear.
\end{proof}
}

\subsection{Grand Loop Argument}
In the following we present a key ingredient towards the proof of Theorem~\ref{thm:basis} that has a similar flavor as Kashiwara's grand loop argument, i.e., we will prove intermediate statements below based on an induction on $\ell$:
\makeatletter\tagsleft@true\makeatother
\begin{align}
\label{eq:Wl}\labeltarget{eq:Wl}&  T_i f_a = f_{\sigma_i(a)} T_i  + f_{\rho_i(a)} \in \Hom_{B^{\otimes d}}(V^\ell, V)
 \textup{ for all } a\in B^{\otimes d}, 1\leq i \leq d-1, \tag{$W[\ell]$}
\\
\label{eq:Ml}\labeltarget{eq:Ml}& f_a f_b = f_{a b}  \in \Hom_{B^{\otimes d}}(V^\ell, V)
 \textup{ for all } a,b \in B^{\otimes d}, \tag{$M[\ell]$}
\\
\label{eq:Ql}\labeltarget{eq:Ql}&  T_i  T_i = f_{S_i} T_i  + f_{R_i} \in \Hom_{B^{\otimes d}}(V^\ell, V)
 \textup{ for all }1\leq i \leq d-1, \tag{$Q[\ell]$}
\\
\label{eq:B2l}\labeltarget{eq:B2l}& T_i  T_j  =  T_j {T_i}  \in \Hom_{B^{\otimes d}}(V^\ell, V)
\textup{ for all } |i- j| >1, \tag{$B_2[\ell]$}
\\
\label{eq:B3l}\labeltarget{eq:B3l}& T_i  T_j  T_i  = T_j  T_i T_j  \in \Hom_{B^{\otimes d}}(V^\ell, V)
\textup{ for all }|i- j| =1, \tag{$B_3[\ell]$}
\\
\label{eq:Rl}\labeltarget{eq:Rl}& { \textup{$f_a \in \Hom_{B^{\otimes d}}(V^\ell, V)
$ does not depend on $r$. }} \tag{$R[\ell]$}
\end{align}

{
Note that the Condition $R[\ell]$ implies that $T_i \in \Hom_{B^{\otimes d}}(V^\ell, V)
$ does not depend on $r$ as well.
}

\makeatletter\tagsleft@false\makeatother
\begin{prop}\label{prop:loop}
{Assume that \eqref{def:wr1}--\eqref{def:qu1} hold, and \eqref{def:br1}--\eqref{def:br5} hold additionally if $d\geq 3$.}
\begin{enumerate}
\item[\textup{(W)}] Condition $W[\ell]$ follows from $R[\ell + 1]$, $M[\ell - 1]$, $W[\ell - 1]${, and $Q[\ell-1]$}.
\item[\textup{(M)}] Condition $M[\ell]$ follows from $W[\ell - 1]$ and $M[\ell - 1]$.
\item[\textup{(Q)}] Condition $Q[\ell]$ follows from combining $Q[\ell - 1]$, $M[\ell - 1]$, and $W[\ell - 1]$.
\item[\textup{(B2)}] Condition $B_2[\ell]$ follows from combining $B_2[\ell - 1]$, $Q[\ell]$, $W[\ell - 1]$.
\item[\textup{(B3)}] Condition $B_3[\ell]$ follows from $B_3[\ell - 1]$, $M[\ell - 1]$, $Q[\ell + 1]$, and {$W[\ell]$}.
\item[\textup{(R)}] Condition $R[\ell]$ follows from combining $R[\ell - 1]$, $B_2[\ell - 2]$, $B_3[\ell - 3]$, and $Q[\ell - 3]$.
\end{enumerate}
\end{prop}
\proof
The proposition follows from an involved calculation that can be found in Appendix \ref{sec:loopproof}. 
\endproof
\begin{lemma}\label{lem:loop}
{Assume that \eqref{def:wr1}--\eqref{def:qu1} hold, and \eqref{def:br1}--\eqref{def:br5} hold additionally if $d\geq 3$.}
Then, $W[\ell], M[\ell], Q[\ell], B_2[\ell], B_3[\ell]$ and $R[\ell]$ hold for all $\ell$.
\end{lemma}
\begin{proof}
The base cases can be verified directly.
For the inductive step, we assume that
$W[i-2], M[i-1], Q[i-1], B_2[i-2], B_3[i-3], R[i-1]$ hold for all $i \leq \ell$.
Therefore, for $n \in \{2,3\}$,
\[
\begin{array}{lllllll}
&W[\ell-2],&M[\ell-1],&Q[\ell-1],&B_n[\ell-n],&R[\ell-1]& \textup{hold by inductive hypothesis,}\\
\Rightarrow&W[\ell-2],&M[\ell-1],&Q[\ell-1],&B_n[\ell-n],&R[\ell]& \textup{hold by Proposition~\ref{prop:loop} (R),}
\\
\Rightarrow&W[\ell-1],&M[\ell-1],&Q[\ell-1],&B_n[\ell-n],&R[\ell]& \textup{hold by Proposition~\ref{prop:loop} (W),}
\\
\Rightarrow&W[\ell-1],&M[\ell],&Q[\ell-1],&B_n[\ell-n],&R[\ell]& \textup{hold by Proposition~\ref{prop:loop} (M),}
\\
\Rightarrow&W[\ell-1],&M[\ell],&Q[\ell],&B_n[\ell-n],&R[\ell]&  \textup{hold by Proposition~\ref{prop:loop} (Q),}
\\
\Rightarrow&W[\ell-1],&M[\ell],&Q[\ell],&B_n[\ell+1-n],&R[\ell]&   \textup{hold by Proposition~\ref{prop:loop} (Bn).}
\end{array}
\]
The proof concludes by proceeding inductively. 
\end{proof}
{
We remark that the proof of Lemma~\ref{lem:loop} work perfectly even for the degenerate case when $d=2$.
Since there is a unique choice of $r$ for $\Sigma_2$, $R[\ell]$ holds for all $\ell$. 
On the other hand, $B_n[\ell]$ are vacuous statements and hence always are true.
The remaining conditions are verified by using the induction above. The grand loop degenerates to a loop that only involves $W[i]$'s, $M[i]$'s, and $Q[i]$'s.
The implications therein are given by Proposition~\ref{prop:loop} (M), (Q), and (W), which in fact  only require \eqref{def:wr1}--\eqref{def:qu1}.
}

\subsection{Proof of the Basis Theorem \ref{thm:basis}}
\begin{proof}
{
First, we} prove that (a) is equivalent to (b). One direction is already taken care of by Proposition~\ref{prop:nec}. For the other direction,
{we will give $V$ a right $A$-module structure, and prove that} the left $B^{\otimes d}$-module homomorphism $\phi : V \to A$, $a \otimes w \mapsto a H_w$ is a right $A$-module isomorphism.

Step 1: We first show that $\phi$ is onto. Note first an arbitrary element $x\in A$ is a sum of elements {of} the form $a_1 \cdots a_n$ for some $n=n(x)$ and each $a_j$ is {either} a generator $H_i$ or an element in $B^{\otimes d}$.
By applying the wreath, quadratic, and braid relations iteratively, $x$ becomes a sum of elements of the form
\eq
b H_{i_1} \cdots H_{i_N}, 
\quad b \in B^{\otimes d},
\quad N \leq n,
\quad s_{i_1} \cdots s_{i_N} \textup{ is a reduced expression},
\endeq
and thus $\phi$ is onto.

Step 2: We show that $\phi$ is an isomorphism. Thanks to Lemma~\ref{lem:loop}, $W[i], M[i], Q[i], B_2[i], B_3[i]$ and $R[i]$ hold for all $i$.
Hence, the assignment $a \mapsto f_a, H_w \mapsto T_w$ induces an algebra homomorphism
$\Phi: A \to \End_{B^{\otimes d}}(V)^{\textup{op}}$,
where $T_w := T_{i_1} \cdots T_{i_N}$ if $w = s_{i_1} \cdots s_{i_N}$ is reduced.
Using this map, we obtain a right $A$-module structure of $V$.

Note that any element in $V$ is of the form $v = \sum_{w\in\Sigma_d} b_w \otimes w$ for some $b_w= b_w(v) \in B^{\otimes d}$. 
It suffices to show that if $\sum_{w\in\Sigma_d} b_w H_w = 0$ then $\sum_{w\in\Sigma_d} b_w\otimes w = 0$. Indeed, 
\begin{align}
\sum_{w\in\Sigma_d} b_w\otimes w 
&= \sum_{w\in\Sigma_d} (b_w\otimes 1)\cdot T_w
= \sum_{w\in\Sigma_d} (1^{\otimes d}\otimes 1)\cdot f_{b_w} T_w
\notag \\
 &= (1^{\otimes d}\otimes 1)\cdot \sum_{w\in\Sigma_d} f_{b_w} T_w
 = (1^{\otimes d}\otimes 1)\cdot \Phi(\textstyle\sum_{w\in\Sigma_d} b_w H_w)
 = 0. 
 \end{align}
{Therefore, (a) and (b) are equivalent.

Next, we show that  (b) and (c) are equivalent. 
For each $w\in \Sigma_d$, let $\underline{w} = (s_{i_1}, \dots, s_{i_n})$ be a fixed reduced expression of $w$, and set
$\sigma_{\underline{w}} := \sigma_{i_1} \circ \dots \circ \sigma_{i_n}$.
By applying the wreath, quadratic, and braid relations iteratively to a typical element $H_w b \in A$, one obtains
\eq\label{eq:bcequiv}
	H_w b \in \sigma_{\underline{w}}(b) H_w + \sum_{x \lt w} B^{\otimes d} H_x,
\endeq
with respect to the Bruhat order on $\Sigma_d$.
By substituting $b$ for $\sigma_{\underline{w}}\inv(b)$, \eqref{eq:bcequiv} becomes
\eq\label{eq:bcequiv2}
	b H_w \in H_w \sigma_{\underline{w}}^{-1}(b) + \sum_{x \lt w} B^{\otimes d} H_x  
	\subseteq H_w \sigma_{\underline{w}}^{-1}(b) + \sum_{x \lt w}  H_xB^{\otimes d},
\endeq
where the latter inclusion follows applying the former part of \eqref{eq:bcequiv2} to every $B^{\otimes d} H_x$, iteratively.
In other words, the transition matrices are invertible since every $\sigma_{\underline{w}}$ is an automorphism. 
Thus, (b) and  (c) are equivalent. 
}
\end{proof}

\subsection{Elias' Basis Theorem}\label{sec:RelBasis}
In the study of Hecke-type categories (see Section~\ref{sec:HC}), the base algebra $B$ is always a commutative ring. In our setup, the base algebra $B$ can be an associative algebra. 
A fundamental question is to determine whether a diagrammatic algebra generated by crossings and dots with relations given by resolving crossings 
has the right size. 
In \cite{E22}, Elias proved a Bergman diamond lemma for the endomorphism algebras (see Section~\ref{sec:HC}) regarding sufficient and necessary conditions
in terms of (conceptual) resolvable ambiguities,
which corresponds to the very first step of our proof of Theorem~\ref{thm:basis}.

In particular, our conditions \eqref{def:wr1}--\eqref{def:br5} describe explicitly the requirements on the choice of parameters in order to resolve the minimal set of ambiguities given in \cite[(5.1), (5.11)]{E22}, except for a couple of them that do not appear in our setup since we require the braid relations to hold in our case.

\section{The Hu algebra $\cA(m)$ and its generalization  $\cH_{m\wr d}$} \label{sec:Hu}
In this section, Hu's original construction of $\cA(m)$ is presented, in which the extra generator $H_1 \in \cH_q(\Sigma_{2m})$ is defined using an implicit procedure via Jucys-Murphy elements $u_m^{\pm}$ of the Hecke algebra of type $B_{2m}$ with unequal parameters $(1,q)$.
Our new result is an explicit formula for $H_1$, which leads to a construction of a bar-invariant basis for $\cA(m)$. 
The explicit formula also makes it possible to construct a generalized Hu algebra $\cH_{m\wr d} \subseteq \cH_q(\Sigma_{md})$.
\subsection{The Construction of the Element $h_{m}$}
Let $d=2m$. In order to quantize \eqref{eq:towerG}, consider the type B Hecke algebra $\cH^\B = \cH_{(1,q)}(W(\B_{2m}))$ of unequal parameters $(1,q)$.
One can realize the Hecke algebras for $\Sigma_m \times \Sigma_m$ and for $\Sigma_{2m}$ as subalgebras of $\cH^\B$ via the generating sets $\{T_1, \dots, T_{m-1}, T_{m+1}, \dots T_{d-1}\}$ and $\{T_1, \dots, T_{d-1}\}$, respectively.
It is well-known that $T_w = T_{i_1} \cdots T_{i_r} \in \cH^\B$ is well-defined if $w = s_{i_1} \cdots s_{i_r} \in W(B_d)$ is a reduced expression.
The tower of groups in \eqref{eq:towerG} can be recovered by taking the specialization $q=1$ from the following tower of algebras:
\eq\label{eq:towerA}
\cH_q(\Sigma_m \times \Sigma_m) 
\subseteq 
\cA(m)
\subseteq
\cH_q(\Sigma_{2m})
\subseteq \cH_{(1,q)}(W(\B_{2m})),
\endeq
where $\cA(m)$ is the Hu algebra which we will define shortly.
From now on, all the computation are carried out inside $\cH^\B$ in the sense of \eqref{eq:towerA}.

\Def
Let $T_{a\to b}$ and $T_{a\to b \to a}$ be the element in $\cH^\B$ corresponding to $s_{a\to b}$ and $s_{a\to b\to a}$ (see \eqref{def:stos}), respectively. 
Denote the Jucys-Murphy elements by
\eq
u_k^{\pm} = \prod_{i=0}^{k-1} (q^{i} \pm T_{i \to 0 \to i}) \in \cH^\B.
\endeq
It is well-defined because its factors commute pairwise.
Note that an empty product is understood as 1 so $u^{\pm}_0 = 1$.
The following result is a useful variant of \cite[Lemma 4.9]{DJM95}.
\endDef
\begin{lem}\label{lem:computeh}
\begin{itemize}
\item[(a)] $u^+_j$ commutes with $T_i$ for $i \in \{ 0,1, \dots, d-1\}\setminus \{j\}$. 
\item[(b)] For $i\geq 1$, $u_i^+ T_{i\to 0} = u_{i+1}^+ T_{1\to i}\inv - q^i u_i^+ T_{1\to i}\inv$.
\item[(c)] Let $x\in \Sigma_{2m}$, $i\geq 1$, and let $\cH := \cH_q(\Sigma_{2m})$. Then
\eq\label{eq:prehm}
u_i^+ T_x T_0\in u_i^+ \cH +\sum_{j > i} \cH u_j^+ \cH.
\endeq
In particular, 
\eq\label{eq:ui0}
u_i^+ T_{i\to 0} \in u_i^+(-q^i T_{1\to i}\inv) + \sum_{j > i} \cH u_j^+ \cH.
\endeq
\end{itemize}
\end{lem}
\proof
Parts (a) and (b) are immediate consequences of the definitions. For part (c), 
write $x = yc \in \Sigma_{2m}$ for $y \in \Sigma_i \times \Sigma_{2m-i}$ and $c \in \Sigma_{2m}$ the shortest representative in the coset $(\Sigma_i \times \Sigma_{2m-i})x$,
so that $u_i^+ T_x = u_i^+ T_y T_c = T_y u_i^+ T_c$ by part (a).
If $c$ fixes 1, then $T_c$ commutes with $T_0$, and hence,  
$$ u_i^+ T_xT_0  = T_y u_i^+ T_0T_c = T_y u_i^+ T_c = u_i^+ T_x$$ 
since $u_i^+ T_0 = u_i^+$ when $i \geq 1$. We are done in this case.

If $c$ does not fix $1$, then $c = s_{i\to1} z$ is a reduced expression for some $T_z$ that commutes with $T_0$, and hence
\begin{align}
\begin{split}
u_i^+ T_x T_0
=
u_i^+ T_y T_{i\to 1} T_z T_0
=
u_i^+ T_y T_{i\to 0} T_z 
=  T_y u_i^+ T_{i\to 0} T_z 
& \quad\textup{by part (a)}
\\
=  T_y (u_{i+1}^+ T_{1\to i}\inv - q^i u_i^+ T_{1\to i}\inv) T_z 
& \quad\textup{by part (b)}
\\
=   T_y u_{i+1}^+ T_{1\to i}\inv T_z - q^i u_i^+ T_yT_{1\to i}\inv T_z.  
& \quad\textup{by part (a)}
\end{split}
\end{align}
\endproof
As a consequence of the lemma,  the definition below makes sense as one can expand $T_{w_{m,m}} u_m^-$ and then apply \eqref{eq:prehm} iteratively to define $h_{m}$. 
\Def\label{def:hm}
Let $h_m$ be the unique element in $\cH:= \cH_q(\Sigma_{2m})$ such that
\eq
u_m^+T_{w_{m,m}} u_m^- \in u^+_m h_m + \sum_{b>m} \cH u_b^+ \cH.
\endeq 
\endDef

\subsection{Definition of $\cA(m)$} \label{sec:Am}
With the construction of $h_m$ one can now define the Hu algebra. 

\Def[Hu algebra]\label{def:Am}
Let $\cA(m)$ be the subalgebra of $\cH_q(\Sigma_{2m})$ generated by 
\eq
T_1, \ldots, T_{m-1}, T_{m+1}, \ldots, T_{2m-1}, H_1,
\endeq
where $H_1= h_m^*$, and $*$ is the unique anti-automorphism of $\cH_q(\Sigma_{2m})$ determined by $T_i^* = T_i$ for all $1\leq i < 2m$. 
\endDef
Several of the important properties that Hu proved about $\cA(m)$ are summarized below. 

\begin{prop}[{\cite[1.7, 1.9, 1.10, 2.2, 2.4]{Hu02}}] \label{prop:Hu}
If $d= 2m \geq 2$, then 
\begin{itemize}
\item[(a)] $\dim \cA(m) = |\Sigma_m \wr \Sigma_2| = 2(m!)^2$. 
\item[(b)] The element $h_m^*$ is invertible if $f_{2m}(q):=	\prod_{i=0}^{2m-1}(1+q^i) \neq 0$.
\item[(c)] $(h_m^*)^2 = h_m^2$, and they are equal to the unique element $z_{m,m} \in Z(\cH_q(\Sigma_m \times \Sigma_m))$ such that
\[
u_m^+T_{w_{m,m}} u_m^- T_{w_{m,m}} u_m^+T_{w_{m,m}} u_m^- =
u_m^+T_{w_{m,m}} u_m^- z_{m,m}.
\]
\item[(d)] $\cA(m)$ admits a presentation with generators $T_1, \dots, T_{m-1}, T_{m+1}, \dots, T_{2m}, H_1$ subject to the braid and quadratic relations  for $T_i$'s as well as the following relations:
\[
H_1^2 = z_{m,m},
\quad
T_i H_1 = 
\begin{cases}
H_1 T_{i+m} & \tif  i < m;
\\
H_1 T_{i-m} &\tif i>m.
\end{cases}
\]
\item[(e)] At the specialization $q=1$, 
$
h_m = 2^m w_{m,m}$,
and
$\cA(m)$ specializes to the group algebra of $\Sigma_{m} \wr \Sigma_2$.

\end{itemize}
\end{prop}
\subsection{A New Formula for $H_{1}$}
In this section, we assume that there is an invertible element $v\in K^\times$ such that $v^{2} = q$. This is true when $K$ is {an algebraically closed field}, and when $K$ is the ring of Laurent {polynomials} in $v$.
The main goal of this section is to give a new and useful interpretation of $H_{1}$.

For simplicity, we consider a renormalized standard basis $\{I_w := v^{-\ell(w)}T_w~|~ w\in \Sigma_{2m}\}$ of $\cH_q(\Sigma_{2m})$ satisfying, for all {$I_i := I_{s_i}$},
\eq
(I_i +v\inv)(I_i-v)= 0,
\quad
\textup{or}
\quad
I_i^2 = (v -v\inv) I_i +1.
\endeq
Denote by $\={\phantom{a}}: \cH_q({\Sigma_{2m}}) \to \cH_q({\Sigma_{2m}})$ the bar involution sending $v$ to $v\inv$ and $I_w$ to $I_{w\inv}\inv$.	
We use the shorthand notation 
\eq
I_{a\to b}^+ := I_{a\to b},
\quad 
I_{a\to b}^- := \={I_{a\to b}} = I_{b\to a}\inv.
\endeq
{Denote by $\{\pm\}^m := \{+, -\}^m$ the set of $m$-tuples of plus or minus signs.}
For $\epsilon = (\epsilon_1, \dots, \epsilon_m) \in \{\pm\}^m$, let
\eq
I_\epsilon = I_{m\to1}^{\epsilon_1}  
I_{m+1\to2}^{\epsilon_2}
\dots
I_{2m-1\to m}^{\epsilon_m},
\endeq 
and let $\Sigma_d$ act on $\{\pm\}^m$ by place permutations.

\begin{prop}\label{prop:braid} 
\begin{itemize}
\item[(a)] Let $\epsilon = (\epsilon_1, \dots, \epsilon_m) \in \{\pm\}^m$. 
If $i<m$, then
\[
I_\epsilon
I_i
= I_{i+m} I_{s_i(\epsilon)}.
\]
\item[(b)] If $a,b >c$, then $I_{a\to c} I_{c\to a}$ commutes with $I_{b\to c} I_{c\to b}$. 
\item[(c)] 
For any $j\in\ZZ$, define an assignment 
\eq\label{def:pijm}
\pi_{jm}: T_i \mapsto T_{i+jm}\quad \textup{for all }i.
\endeq
Let $c_{m,0}:=1$, and let $c_{m,i}:= \pi_{m}(I_{i\to1} I_{1\to i}) = I_{m+i \to m+1}I_{m+1 \to m+i}$ for $i\geq 1$.
Then $c_{m,i}$ commutes with $c_{m,j}$ for all $i,j \geq 1$. 
Moreover, the following product is well-defined:
\eq
C_{\epsilon} = \prod_{i; \epsilon_i = -} c_{m,i-1}.
\endeq
\end{itemize}
\end{prop}

\proof
(a) This is a direct consequence of the braid relations.
One obtains a graphical proof by treating $I_i$ and $I_i\inv$ as opposite crossings of the $i$th and the $(i+1)$th strings.

(b) It suffices to prove the case $a=b+1$. Note that
\eq
I_{a\to c} I_{c\to a} I_b
=
I_{a\to c} I_{c\to b-1} I_b I_a I_b
=
I_{a\to c} I_{c\to b-1} I_a I_b I_a
=
I_a I_b I_a I_{b-1 \to c} I_{c\to a}
= I_b I_{a\to c} I_{c\to a},
\endeq
it then follows by an inductive argument that $I_b, \dots, I_c$ (and hence $I_{b\to c} I_{c\to b}$) commutes with $I_{a\to c} I_{c\to a}$.

(c) The first assertion follows from Part (b) since $c_{m,i} c_{m,j} = \pi_m (I_{i\to 1}I_{1\to i} I_{j\to 1} I_{1\to j})$. The second assertion is an immediate consequence of the first one. 
\endproof

\begin{lem}
Let $h_{m,i}$ be the unique element in $\cH:=\cH_q(\Sigma_{2m})$ such that
\eq
u_m^+ T_{w_{m,i}} u_i^- \in u_m^+ h_{m,i} + \sum_{b>m} \cH u_b^+ \cH.
\endeq
Then, the elements $h_{m,i} (1\leq 1 \leq m)$ are given by the formula below recursively:
\eq
\begin{split}
h_{m,1} &=  v^m(I_{m \to 1} + I^-_{m \to 1}),
\\
h_{m,i+1} &= v^{m+2i} (h_{m,i}I_{i+m \to i+1} +  c_{m,i} h_{m,i} I^-_{i+m \to i+1}).
\end{split}
\endeq
In particular,  $h_m = h_{m,m}$ is given by
\eq
h_m = v^{m(2m-1)}\sum_{\epsilon \in \{+,-\}^m} C_\epsilon 
I_{m\to1}^{\epsilon_1}  
I_{m+1\to2}^{\epsilon_2}
\dots
I_{2m-1\to m}^{\epsilon_m}.
\endeq
\end{lem}
\proof
The base case $i=1$ follows immediately from Lemma~\ref{lem:computeh}(c).
For the inductive step, by  definition of $u_i^-$, we have
\eq\label{eq:goal}
u_m^+ T_{w_{m,i+1}} u_{i+1}^- = q^{i} u_m^+ T_{w_{m,i+1}} u_i^- - u_m^+ T_{w_{m,i+1}} T_{i\to0\to i} u_i^-.
\endeq
The first term on the right hand side of \eqref{eq:goal} is
\eq
 q^{i} u_m^+ T_{w_{m,i+1}} u_i^-  =  q^{i} u_m^+ T_{w_{m,i}} u_i^- T_{m+i\to i+1}
 = u_m^+ q^{i} h_{m,i}  T_{m+i\to i+1} + \sum_{b>m } \cH u_b^+ \cH,
\endeq
and hence its contribution to $h_{m,i+1}$ is $ q^{i} h_{m,i}  T_{i+m\to i+1} = v^{m+2i} h_{m,i}  I_{i+m\to i+1}$. 
The second term on the right hand side of \eqref{eq:goal} is
\eq
\begin{array}{ll}
 - u_m^+ T_{w_{m,i+1}} T_{i\to1} T_{0\to i} u_i^-
\\
\quad =
  - u_m^+ T_{i+m\to1+m} T_{w_{m,i+1}} T_{0\to i} u_i^-
  &\textup{Proposition }\ref{prop:braid}(a)
  \\
\quad=
  - T_{m+i\to m+1} u_m^+  T_{m\to 1} (T_{m+1\to 2} \dots T_{i+m \to i+1}) T_{0\to i} u_i^-
 &\textup{Lemma }\ref{lem:computeh}(a)
 \\ 
 \quad  =
  - T_{m+i\to m+1} u_m^+  T_{m\to 0} (T_{m+1\to 2} \dots T_{i+m \to i+1}) T_{1\to i} u_i^-
  \\
  \quad=   q^m T_{m+i\to m+1} u_m^+  T^-_{m\to 1} T_{m+1\to 2} \dots T_{i+m \to i+1} T_{1\to i} u_i^- + \sum_{b>m } \cH u_b^+ \cH
&\textup{Lemma }\ref{lem:computeh}(b)
\\
  \quad =   q^m T_{m+i\to m+1}T_{m+1\to m+i} u_m^+  T_{w_{m,i}} T^-_{i+m \to i+1}  u_i^- + \sum_{b>m } \cH u_b^+ \cH 
&\textup{Proposition }\ref{prop:braid}(a)
\\ 
\quad  =   q^m c_{m,i} u_m^+  T_{w_{m,i}}u_i^-  T^-_{i+m \to i+1}    
 + \sum_{b>m } \cH u_b^+ \cH
 &\textup{Lemma }\ref{lem:computeh}(a)
 \\
  \quad=  q^m c_{m,i}  u_m^+   h_{m,i} T^-_{i+m \to i+1}      + \sum_{b>m } \cH u_b^+ \cH,
  \end{array}
\endeq
and hence its contribution to $h_{m,i+1}$ is $v^{m+2i} c_{m,i} h_{m,i} I^-_{i+m \to i+1}$, and we are done. 
\endproof
\begin{theorem}
We have an explicit formula for $H_1 = H_1(m) :=h_m^*$ as follows:
\eq\label{eq:h*}
H_1 = v^{m(2m-1)}\sum_{\epsilon \in \{+,-\}^m}  
I_\epsilon C_\epsilon,
\quad
\textup{where}
\quad
I_\epsilon = I_{m\to 2m-1}^{\epsilon_1}
\dots
I_{1\to m}^{\epsilon_m},
\quad
C_\epsilon = \prod_{i;\epsilon_i=-}c_{m,m-i}.
\endeq
\end{theorem}
\exa\label{ex:h2}
Consider  $H_1 \in \cA(2) \subseteq \cH_q(\Sigma_4)$. 
We use {abbreviations} such as $\={T}_2 T_1 \={T}_3 \={T}_2 := T_{\overline{2}.1.\overline{3.2}}$.
Then
$h_{2} = v^6 (I_{2.1.3.2}  
+  I_{2.1.\overline{3.2}}
+ I_3^2(I_{\overline{2.1}.3.2} + I_{\overline{2.1}.\overline{3.2}}){)}$,
or 
\eq
H_1 = v^6( I_{2.3.1.2} 
+  I_{2.3.\overline{1.2}} 
+ (I_{\overline{2.3}.1.2} + I_{\overline{2.3}.\overline{1.2}}) I_3^2{)},
\endeq
While $I_{2.1.3.2}^2$ consists of summands involving $I_2$ in terms of the standard basis; its alternative, $H_1$, squares into the parabolic subalgebra $\cH_q(\Sigma_2 \times \Sigma_2) = \<I_1, I_3\>$. 
To be precise,
{we obtained an explicit formula for the central element $z_{2,2} = H_1^2$ (see Proposition~\ref{prop:Hu}(c)):}
\eq
\begin{split}
z_{2,2} &=v^6(q-1)^2(q^4+2q^3-2q^2+2q+1) I_{31}
\\
&+v^7(q-1)(q^4+4q^3-2q^2+4q+1)(I_1+I_3)
\\
&+2v^8(q^4+4q^3-2q^2+4q+1).
\end{split}
\endeq
We currently have no explanations for these coefficients appearing in $z_{m,m}$.
\endexa
\subsection{Bar-Invariant Bases}
In this section, we assume $K = \ZZ[v^{\pm1}]$. The element $H_1$ is defined via solving an equality (see Definition~\ref{def:hm}) that involves Jucys-Murphy elements in the type B Hecke algebra of unequal parameters. There seems to be no a priori reason for $H_1$ to admit an explicit formula in Hecke algebra of Type A. 
In this section, we show miraculously that $H_1$, after {being} multiplied by an element arising from the longest element, leads to a highly nontrivial bar-invariant element $b_1$, 
and thus to a bar-invariant basis for the Hu algebra.
The element $b_1$ seems to afford a positive expansion in terms of the dual canonical basis of the ambient Hecke algebra, and in turn affords a potential theory of geometric realization and/or categorification for such wreath product $\Sigma_m \wr \Sigma_d$ that are not Coxeter groups in general.

\begin{prop}\label{prop:gamma}
Recall $\pi_m$ from \eqref{def:pijm}.
Let $C := C_{(-,-,\dots,-)}$, and let $\gamma :=  \pi_m(I_{\wo{m}})$ where  $\wo{m} \in \Sigma_m$ is the longest element. 
Then:
\begin{itemize} 
\item[(a)] $I_{\wo{m}} = I_{\wo{m-1}} I_{m-1\to 1}$.
\item[(b)] $I_{\wo{m}} I_i = I_{m+1-i} I_{\wo{m}}$.
\item[(c)] $C = \gamma^2$.
\end{itemize}
\end{prop}
\proof
Part (a) follows from {the fact} that $s_1 s_{2\to1} \dots s_{m-1\to1} = \wo{m}$ is a reduced expression.
Part (b) follows from {the fact} that $\wo{m} s_i = s_{m+1-i} \wo{m}$ for all $i$.
For part (c), we prove it by an induction on $m$. The base case $m=2$ follows since $\wo{2} = s_1$ and $C= I_3^2$ as computed in Example~\ref{ex:h2}.
For $m \geq 3$, we have 
\eq
\begin{array}{rll}
C &= c_{m,1} \dots c_{m,m-1} = \pi_m(I_1^2 I_{2.1.1.2} \dots I_{m-1\to1} I_{1\to m-1})
\\
&= \pi_m(I_{\wo{m-1}}^2 I_{m-1\to1} I_{1\to m-1}) &\textup{via the induction hypothesis}
\\
&= \pi_m (I_{\wo{m-1}} I_{\wo{m}} I_{1\to m-1}) &\textup{by part (a)}
\\
&= \pi_m (I_{\wo{m-1}} I_{m-1\to 1}I_{\wo{m}} ) &\textup{by part (b)}
\\
&= \gamma^2. &\textup{by part (a)}
\end{array}
\endeq
\endproof
The extra generator $H_1$ behaves well with respect to the bar involution in the following way:
\begin{prop}\label{prop:cb1}
The following element  is bar-invariant:
\eq
b_1:= v^{-m(2m-1)} H_1 \={\gamma} = \sum_{\epsilon \in \{\pm\}^m} I_\epsilon C_\epsilon \={\gamma}.
\endeq
\end{prop}
\proof
Note that for any $\epsilon \in \{\pm\}^m$,
\eq
\={I}_\epsilon
=
\={I^{\epsilon_1}_{m\to 2m-1}} \dots \={I^{\epsilon_m}_{1\to m}}
=
I^{-\epsilon_1}_{m\to 2m-1} \dots I^{-\epsilon_m}_{1\to m}
=
I_{\={\epsilon}},
\endeq
where $\={\epsilon} := -\epsilon \in \{\pm\}^m$. 
Recall that $H_1 = v^{m(2m-1)}\sum_\epsilon I_\epsilon C_\epsilon$,
so
$b_1 = \sum_\epsilon I_\epsilon C_\epsilon \={\gamma},$
and hence
\eq
\={b}_1 =  \sum_\epsilon \={I}_\epsilon \={C}_\epsilon \gamma
= \sum_{\={\epsilon}} I_{\={\epsilon}} \={C}_{\epsilon} \gamma
= \sum_{\epsilon} I_{\epsilon} \={C}_{\={\epsilon}} \gamma.
\endeq
Therefore, $b_1$ is bar-invariant as long as
\eq\label{eq:Cge}
C_{\epsilon} \={\gamma} = \={C}_{\={\epsilon}}\gamma
\quad
\textup{ for all }\epsilon \in \{\pm\}^m.
\endeq
Thanks to Proposition~\ref{prop:gamma}(c), one has $\gamma^2 = C = C_{\={\epsilon}} C_{\epsilon}$ for any $\epsilon$, and hence \eqref{eq:Cge} follows from a right multiplication by $\={\gamma}$ and a left multiplication by $\={C}_{\={\epsilon}}$.

\endproof
{With the help of this distinguished element $b_1$, 
we are now able to construct a bar-invariant basis $\{b_x~|~x\in \Sigma_m\wr \Sigma_2\}$ of the Hu algebra that satisfies a unitriangular relation with respect to a standard basis $\{I_x ~|~ x\in \Sigma_m \wr\Sigma_2\}$.
See Remark~\ref{rmk:HuCB} for a discussion on the positivity of $\{b_x\}_x$ in comparison with the canonical bases.}
\begin{thm}\label{thm:std}
Let $I_{wt_1} := I_w b_1$ for $w\in \Sigma_m \times \Sigma_m$.
Then $\{I_x ~|~ x\in \Sigma_m \wr \Sigma_2\}$ forms a standard basis of $\cA(m)$ in the sense that $\={I_x} \in I_x + \sum_{y<x} \ZZ[v^{\pm1}] I_y$ for all $x\in \Sigma_m \wr \Sigma_2$.
\end{thm}
\proof
It follows from {the fact} that $\{I_w ~|~ \Sigma_m \times \Sigma_m\}$ forms a standard basis for $\cH_q(\Sigma_m\times \Sigma_m)$ with respect {to} the Bruhat order inherited from $\Sigma_{2m}$.
In particular, for $x = wt_1$, we have
$\={I_x} = \={I_w} b_1 \in I_{x t_1} + \sum_{y<x} \ZZ[v^{\pm1}] I_{y t_1}$.
\endproof

\begin{thm}
Let $\{b_w ~|~ \Sigma_m \times \Sigma_m\}$ be the canonical basis for $\cH_q(\Sigma_m\times \Sigma_m)$ in the sense that
$b_w \in I_w + \sum_{y<w} v\ZZ[v] I_y$,
and let
$b_{wt_1} := b_w b_1$ for $w\in \Sigma_m \times \Sigma_m$.
Then $\{b_x ~|~ x\in \Sigma_m \wr \Sigma_2\}$ forms a bar-invariant basis of $\cA(m)$ satisfying 
\eq
b_x \in I_x + \sum_{y < x} v\NN[v] I_y, 
\endeq
with respect to the Bruhat order inherited from $\Sigma_{2m}$.
\end{thm}
\proof
The bar-invariant part is due to Proposition~\ref{prop:cb1}. 
The unitriangular part follows from {the fact} that
$b_{xt_1} = b_x b_1 \in I_{xt_1} + \sum_{y<x} v\NN[v] I_{yt_1}$.
\endproof

\rmk\label{rmk:HuCB}
While the canonical basis of $\cH_q(W)$ of any Weyl group $W$ admits positive structural constants, i.e., $b_x b_y \in \sum_z \NN[v^{\pm1}] b_z$; our bar-invariant basis does not have such positivity because $b_1^2$ has mixed signs already.

However, $b_1$ seems to admit a negative expansion in the canonical basis, equivalently, a positive expansion in the dual canonical basis $\{c_w ~|~ w\in \Sigma_{2m}\}$, i.e.,
\eq
b_1 \in \sum_{z \in \Sigma_{2m}} \NN[(-v^{-1})^{\pm1}] b_z = \sum_{z \in \Sigma_{2m}} \NN[v^{\pm1}] c_z,
\endeq 
where the dual basis element $c_w$ is characterized by $\={c}_w = c_w \in I_w + \sum_{y<w} (-v^{-1})\NN[-v^{-1}] I_y$, and it satisfies that $c_x c_y \in \sum_z \NN[-v^{\pm1}] c_z$.
\endrmk

\subsection{} We present a concrete calculation of $b_{1}$ in a specific example below.

\exa
For $H_1 = H_1(3)$, write $I_{\epsilon_1 \epsilon_2}:= I_{(\epsilon_1,\epsilon_2,+)} + I_{(\epsilon_1,\epsilon_2,-)}$ for short. 
Then
$H_1 = v^{15}(I_{++} + I_{+-}I_{4.4} + I_{-+}I_{5.4.4.5} + I_{--}I_{4.4.5.4.4.5})$, 
$\gamma =  I_{4.5.4}$, and hence 
\eq
\begin{split}
b_1 &= (I_{++} + I_{+-}I_{4.4} + I_{-+}I_{5.4.4.5} + I_{--}I_{4.4.5.4.4.5})I_{\={4.5.4}}
\\
&= (I_{++}I_{\={4.5.4}} + I_{+-}I_{4.\={5.4}} + I_{-+}I_{\={4}.5.4} + I_{--}I_{4.5.4}) .
\end{split}
\endeq
The element $b_1$ is indeed bar-invariant since
\eq
\begin{split}
\={b}_1 &= (I_{--} + I_{-+}I_{\={4.4}} + I_{+-}I_{\={5.4.4.5}} + I_{++}I_{\={4.4.5.4.4.5}}) I_{4.5.4}
\\
&=  (I_{--}I_{4.5.4} + I_{-+}I_{\={4}.5.4} + I_{+-}I_{4.\={5.4}} + I_{++}I_{\={4.5.4}}) 
= b_1.
\end{split}
\endeq

Examples for a positive expansion of $b_1=b_1(m)$ in the dual canonical basis are shown below:
for $m=1$, 
\eq
b_1(1) = 2b_{s_1} + (-v-v\inv) = 2c_{s_1} + (v+v\inv).
\endeq
For $m=2$, we abbreviate $c_{s_{i}s_{j}\dots s_{k}}$ by $c_{i.j.\dots.k}$, and obtain: 
\eq
\begin{split}
b_1(2) &
= 4c_{2.3.1.2.3} 
+ 2(v+v\inv)(c_{2.3.1.2}+c_{2.1.2.3}+c_{3.1.2.3})
\\
&+ 4(c_{2.1.2}+c_{3.1.2})
+2(v^2+v^{-2})c_{1.2.3}
+2(v+v\inv)(c_{2.1}+c_{1.2}+c_{3.2}+c_{2.3})
\\
&+(v^2+2+v^{-2})(c_1+2c_2+c_3)
+(v^3+v+v\inv+v^{-3}).
\end{split}
\endeq

\endexa

\subsection{Generalized Hu Algebra} 

With the explicit formula above, we are in a position to define a Hecke subalgebra $\cH_{m\wr d}$ of $\cH_q(\Sigma_{md})$ for the wreath product $\Sigma_m \wr \Sigma_d$.
Note that Hu's Hecke subalgebra is a special case $\cA(m) = \cH_{m \wr 2}$.
\begin{Def}
Let $\cH_{m\wr d}$ be the subalgebra of $\cH_q(\Sigma_{md})$ generated by
\[
T_i, ( i \in [md]\setminus m\ZZ), \quad 
H_j, (1\leq j \leq d-1),
\]
where $H_{j+1} = \pi_{jm}(h^*_m)$ (see \eqref{def:pijm}).
The algebra $\cH_{m \wr d}$ will be called the {\em generalized Hu algebra} with $\cH_{m \wr 2} = \cA(m)$.
\end{Def}

The braid relations fail for the elements $H_i$'s.
For example, in the special case when $m=1$, the $H_i$'s satisfy the modified braid relations
\eq
H_i H_{i+1} H_i - H_{i+1} H_iH_{i+1} = (q-1)^2 (H_{i+1} - H_i),
\endeq
which appeared in the non-standard presentation of the type A Hecke algebra (cf. \cite{Wa07}).
In other words, Wang's non-standard presentation of $\cH_q(\Sigma_d)$  is just the ``standard'' presentation of the generalized Hu algebra $\cH_{1\wr d}$.
This suggests that there could be a general theory for covering algebras for $\cH_{m\wr d}$.

\section{The Schur Algebras for Quantum Wreath Products} \label{sec:SchurW}
\subsection{} 
In this section, 
we investigate the connection between various properties that are referred as the ``Schur-Weyl duality'', including the double centralizer property, the existing of a splitting, and the existence of a certain idempotent.
A systematical approach is provided to obtain the Schur-Weyl duality for the quantum wreath product  $A = B \wr \cH(d)$ arising from a Schur-Weyl duality for the base algebra $B$.
As a consequence, the Schur-Weyl duality for the Hu algebra $\cA(m) = \cH_q(\Sigma_m) \wr \cH(2)$ is established for the first time.
It should be noted that our approach is quite encompassing and allows us to recover the Schur-Weyl duality for other Hecke-like algebras appearing in Section~\ref{sec:examples}.

\begin{Def}[Wreath Schur Algebra]
Assume that $T$ is a right $A$-module.
We call $S^A(T) := \End_A(T)$ the {\em wreath Schur algebra}.
\end{Def}
\subsection{} \label{sec:SF}
Various constructions have appeared in the literature that have involved endomorphism algebras and the double centralizer property (cf. \cite[Chapter 1]{CPS96} \cite[Sections 2-3]{DEN04}). 
In familiar cases like commuting actions of the general linear group and the symmetric group, the ``Schur functor" is constructed via an explicit idempotent. 
For more general settings,  idempotents in algebras are not easy to locate. For the purposes of this paper, we present an alternative approach via splittings to prove 
the existence of an idempotent $e\in S$ (where $S$ is an endomorphism algebra) that can be used to construct a functor from $\text{Mod}(S)$ to $\text{Mod}(eSe)$.

For now, assume that $A$ is an arbitrary associative $K$-algebra, and $T$ is a right $A$-module with $S=\End_{A}(T)$. 
Then $T$ is a left $S$-module (under composition), 
and $T$ is an $S$-$A$-bimodule. Suppose one has a splitting,  as right $A$-modules,
\begin{equation} \label{E:splitW}
T\cong A\oplus Q
\quad
(\textup{for some right }A\textup{-module }Q).
\end{equation} 
There exists a contravariant functor 
${\mathcal G}(-):=\text{Hom}_{A}(-,T)$ whose image is a left $S$-module. Under this functor, 
${\mathcal G}(A)=\text{Hom}_{A}(A,T)\cong T$ is a direct {summand} of $S$ as a left $S$-module by using the splitting. 
Consequently, $T$ is a projective left $S$-module.  

Under the condition \eqref{E:splitW}, 
there exists an idempotent $e\in S$ such that $T\cong Se$. 
Moreover, 
$$\End_{S}(T)\cong \End_{S}(Se)\cong eSe.$$
One can define a Schur functor ${\mathcal F}:\text{Mod}(S) \rightarrow \text{Mod}(eSe)$ via ${\mathcal F}(M)=eM$.  
In this setting, the theory of Schur functors and their inverses has 
been developed in \cite{DEN04}.
\subsection{} The question remains when one can identify $A$ with $eSe$. 
When this holds the aforementioned Schur functor will conveniently take objects in $\text{Mod}(S)$ to $\text{Mod}(A)$. 
When $A$ is a self-injective algebra the following result is a corollary of a theorem proved by K{\"o}nig, Slungard and Xi \cite{KSX01}. 
The original statement of the corollary first appeared in \cite[59.6]{CR62}.  

\begin{theorem} \label{thm:SchurEquiv}
Let $S=\End_{A}(T)$. Assume that 
\begin{itemize} 
\item[(a)] $A$ is a finite dimensional self-injective algebra,
\item[(b)] $T$ is a faithful $A$-module.
\end{itemize}
\end{theorem}
Then  $\End_{S}(T)\cong A$. 
\vskip .25cm 
Condition (a) of the theorem will be satisfied when $A$ is a quantum wreath product as long as the base algebra $B$ is a symmetric algebra and the assumptions in Proposition~\ref{prop:struc}(b) hold. 
For examples of quantum wreath products, to ensure the existence of a natural Schur functor, 
one needs to establish a splitting of $T$ as in (\ref{E:splitW}). 
This will ensure that condition (b) of the theorem is satisfied. 

\subsection{A Splitting Lemma, I}\label{sec:SchurA}
{For applications in Schur-Weyl duality, we make the following assumptions:
\eq\label{eq:SWcond}
\sigma = \textup{flip}, 
\quad\rho = 0, 
\quad\textup{and }
\eqref{def:wr1} \textup{--} \eqref{def:br5}\textup{ hold}.
\endeq
}
{Within Section~\ref{sec:SchurA}, we also assume that 
\eq\label{eq:repsilon}
r^2 =  \epsilon(S)r + \epsilon(R)
\quad \textup{for some}
\quad r\in K.
\endeq
}

Let $V_B$ be the vector space over $K$ with basis $\{v_i ~|~ i \in I\}$, where $I$ is a (possibly infinite) poset {that is totally ordered}.
Assume that $V_B$ is a right $B$-module, and hence $V_B^{\otimes d}$ is a right module over $B^{\otimes d}$ via the factorwise action.
For $\mu = (\mu_j)_j \in I^d$, set $v_\mu = v_{\mu_1} \otimes \ldots \otimes v_{\mu_d} \in V_B^{\otimes d}$. 
{Assume that there is $\mu \in I^d$ such that $\mu_1 < \dots < \mu_d$, and let  $W = W(\mu)$ be the subspace of $V_B^{\otimes d}$ spanned by elements of the form $v_{\mu\cdot g}$, where $\Sigma_d$ acts on $I^d$ by place permutations.}
Let the Hecke-like generators $H_i$'s in $A = B \wr \cH(d)$ act on the $B^{\otimes d}$-submodule $W$ from the right by
\eq\label{eq:vTiA}
v_\mu. H_i = \Bc{
v_{\mu \cdot s_i}&\tif \mu_i< \mu_{i+1};
\\
 v_{\mu \cdot s_i}.{R_i} + v_\mu.{S_i}&\tif \mu_i > \mu_{i+1}.
}
\endeq
{Thanks to the assumption \eqref{eq:repsilon} and that $I$ is totally ordered, the action \eqref{eq:vTiA} on $W$ extends to the following action on $V_B^{\otimes d}$, as long as \eqref{eq:vTiA} is well-defined: 
\eq\label{eq:vTiA2}
v_\mu. H_i = \Bc{
v_{\mu \cdot s_i}&\tif \mu_i< \mu_{i+1};
\\
r v_\mu &\tif \mu_i = \mu_{i+1};
\\
 v_{\mu \cdot s_i}.{R_i} + v_\mu.{S_i}&\tif \mu_i > \mu_{i+1}.
}
\endeq
}
{We now show that \eqref{eq:vTiA} is well-defined  in {two} separate cases. The first case is when $B$ acts on $V_B$ (and hence on the tensor products of $V_B$)  by a counit $\epsilon:B \to K$.
 The second case is the degenerate case when $A$ is just a twisted tensor product of $B^{\otimes d}$ by a group algebra $K\Sigma_d$.
}
\prop\label{rmk:Haction}
Assume that  \eqref{eq:SWcond} holds. Then:{
\begin{enumerate}[(a)]
\item  Equation~\eqref{eq:vTiA} extends to an $A$-action on $W$ if $B$ acts on $V_B$ by a counit $\epsilon$. 
\item If $S=0, R= 1\otimes 1$, then \eqref{eq:vTiA} extends to an $A$-action on $W$ .
\end{enumerate}
}
\endprop
\proof
{For (a), the verification of the quadratic relations reduces to the rank two case: for $i < j$,
$(v_{(i,j)} H_1) H_1 = v_{(j,i)} H_1 = \epsilon(S) v_{(j,i)} + \epsilon(R) v_{(i,j)} =  v_{(i,j)} (H_1^2)$. The other case $i > j$ can be checked similarly.
The verification of the braid relations reduces to the rank three case. We only present the most complicated case, for $i < j < k$:
\eq 
\begin{split}
&v_{(k,j,i)}H_1 H_2 H_1 = (\epsilon(S) v_{(k,j,i)} + \epsilon(R) v_{(j,k,i)})H_2H_1
\\ 
&\quad
=(\epsilon(S^2) v_{(k,j,i)} 
+ \epsilon(SR)(v_{(j,k,i)} + v_{(k,i,j)}) 
+  \epsilon(R^2)v_{(j,i,k)})H_1
\\
&\quad
=\epsilon(S^3+SR) v_{(k,j,i)} + \epsilon(S^2R)(v_{(j,k,i)} + v_{(k,i,j)}) + \epsilon(SR^2)(v_{(i,k,j)} + v_{(j,i,k)}) + \epsilon(R^3)v_{(i,j,k)} 
\\
&\quad
=(\epsilon(S^2) v_{(k,j,i)} 
+ \epsilon(SR)(v_{(j,k,i)} + v_{(k,i,j)}) 
+  \epsilon(R^2)v_{(i,k,j)})H_2
= v_{(k,j,i)}H_2 H_1 H_2. 
\end{split}
\endeq
When $i=j$, the braid relations hold due to a case-by-case analysis similar to the ones for $i\neq j$. The verification of quadratic relations  reduces to the rank two case: 
$$(v_{(i,i)} H_1) H_1 = r^2 v_{(i, i)} = \epsilon(S) v_{(i,i)} + \epsilon(R) v_{(i,i)} =  v_{(i,i)} (H_1^2).$$ 

The wreath relations hold since $\rho = 0$ and the fact that both $a\otimes b, b\otimes a \in B\otimes B$ act on $V_B^{\otimes 2}$ by the same scalar multiple $\epsilon(a)\epsilon(b)$. 


For (b), the $A$-action on $W$ is given explicitly by $v_\mu. H_i = v_{\mu\cdot s_i}$. The braid and quadratic relations are immediate; while the wreath relations follow from the rank two verification: 
$$v_{(i,j)} \sigma(b) H_1 = (\sum_{k} v_i b''_k \otimes v_j b'_k) H_1 = v_{(j,i)} b = v_{(i,j)} H_1 b,$$ 
where $b = \sum_k b'_k \otimes b''_k \in B\otimes B$.
}
\endproof
Schur-Weyl duality in the literature often assumes a condition $n \geq d$ for the double centralizer property to hold. In the following lemma, we demonstrate that the condition $n\geq d$ is essentially a shadow of the assumption that $V_B$ contains $d$ distinct direct summands which are isomorphic to $B$ as $B$-modules.
To be precise, one needs to assume additionally the subspace $W$ obtained from these direct summands admits an $A$-action given by \eqref{eq:vTiA} and right multiplications by $B^{\otimes d}$ (e.g., the {two} cases in Proposition~\ref{rmk:Haction}).
In such a case, $W \cong A$ as $A$-modules, and hence the desired splitting is obtained when the $A$-action on $W$ extends to $V_B^{\otimes d}$. 
\begin{lem} \label{lem:QSchur}
Assume that \eqref{eq:SWcond} holds, and there are right $B$-module isomorphisms $\phi_j: U_j \to B, v_{\mu_j} b \mapsto b$ such that $U_j$'s are distinct direct summands of $V_B$. 

{
If \eqref{eq:vTiA} extends to an $A$-action on $W=W(\mu)$,
then $W \simeq A$ as right $A$-modules. 
If \eqref{eq:vTiA}extends further to an $A$-action on  $V_B^{\otimes d}$,} 
then $V_B^{\otimes d}$ has the following splitting, as right $A$-modules,
\[
V_B^{\otimes d} \simeq A \oplus Q.
\]
\end{lem}
\proof
We may assume that ${\mu_1} < {\mu_2} < \dots < {\mu_d}$.
It suffices to show that the assignment 
\eq
\psi:v_\mu.a \mapsto a
\quad
(\textup{for all }a \in A)
\endeq
defines an $A$-module isomorphism $W \to A$.
{We show first that $\psi$ is well-defined. Suppose that $v_\mu H_w b = v_\mu H_{w'}b'$ for some $w, w'\in \Sigma_d, b,b' \in B^{\otimes d}$. By \eqref{eq:vTiA}, 
$v_\mu H_w = v_{\mu\cdot w} \in U_{w(1)} \otimes \cdots U_{w(d)}$, so is $v_\mu H_w b$ because $U_j$'s are summands. Since $U_j$'s are all distinct, $w$ must agree with $w'$. Therefore, $b = (\phi_{w(1)} \otimes \dots \otimes \phi_{w(d)})(v_{\mu\cdot w} b) =(\phi_{w(1)} \otimes \dots \otimes \phi_{w(d)})(v_{\mu\cdot w} b') = b'$.}

Next, it is easy to verify that $\psi$ is an epimorphism.
For injectivity, thanks to the basis theorem, an arbitrary element in $A$ is of the form $a= \sum_{w\in \Sigma_d} H_w (b_1^{(w)} \otimes \dots \otimes b_d^{(w)})$ for some $b^{(w)}_i \in B$. 
Since ${\mu_1} < {\mu_2} < \dots < {\mu_d}$, by \eqref{eq:vTiA}, $H_w$ acts on $v_{({\mu_1}, {\mu_2}, \dots, {\mu_d})}$ by place permutation.

Suppose that $v_{({\mu_1}, {\mu_2}, \dots, {\mu_d})}.\sum_{w\in \Sigma_d} H_w (b_1^{(w)} \otimes \dots \otimes b_d^{(w)}) = 0$.
Each $v_{({\mu_1}, {\mu_2}, \dots, {\mu_d})} H_w (b_1^{(w)} \otimes \dots \otimes b_d^{(w)})$ lies in a different summand $U_{w(\mu_1)} \otimes \dots \otimes U_{w(\mu_d)} \subseteq V_B^{\otimes d}$, and hence $b_1^{(w)} \otimes \dots \otimes b_d^{(w)}$ must be zero for all $w$.
\endproof
Finally, a double centralizer property is obtained in the context of Theorem~\ref{thm:SchurEquiv}, where the $S$-$A$-bimodule $T$ {is $V_B^{\otimes d}$.}
\begin{cor}
{Let $A = B \wr \cH(d)$ be a quantum wreath product for a finite dimensional algebra $B$ such that  Proposition~\ref{prop:struc}(b)  holds.
Assume that Lemma \ref{lem:QSchur} holds.
Then, the Schur duality holds, i.e., $\operatorname{End}_{S^A(T)}(T)\cong A$, where {$T = V_B^{\otimes d}$}.}
{ Moreover, there exists a Schur functor.}
\end{cor}
\proof
The statement follows by combining Theorem~\ref{thm:SchurEquiv} and Lemma~\ref{lem:QSchur}.
{The existence of a Schur functor follows from the discussion in Section \ref{sec:SF}, thanks to the existence of the splitting of $T$.}
\endproof
\subsection{Examples}
\subsubsection{Hecke Algebras of Type A}
We retain the setup as in Section~\ref{ex:AK}, and then set $m=1$. 
Therefore, $B=K$ and $A=B \wr \cH(d)$ is just the Hecke algebra of type A.
{Note that in this degenerate case, $X$ is identified with $q_1 \in K^\times$ and $\rho = 0$.}
Let $V_B=V(n)$ be the $K$-space spanned by {$v_i$ where $i \in I := \{1, \dots, n\}$ with the usual total order.
Thus,  $V_B^{\otimes d}$ admits an $A$-action by Proposition~\ref{rmk:Haction}(a).} To be precise, for $1\leq i \leq d-1$,
\eq\label{eq:vHi}
v_\mu. H_i = \Bc{
v_{\mu \cdot s_i}&\tif \mu_i< \mu_{i+1};
\\
q v_{\mu \cdot s_i} &\tif \mu_i= \mu_{i+1};
\\
qv_{\mu \cdot s_i} + (q-1)v_\mu&\tif \mu_i > \mu_{i+1}.
}
\endeq

The assumptions in Lemma~\ref{lem:QSchur} hold if and only if $V(n)$ contains a $d$-dimensional subspace, i.e., $n\geq d$, which is the well-known condition to afford the quantum Schur-Weyl duality \cite{J86, DJ89}.
Moreover, $S^A(V_B^{\otimes d})$ coincides with the type A $q$-Schur algebra $S_q(n,d) := \End_{\cH_q(\Sigma_d)}(V(n)^{\otimes d})$.
\begin{cor}
If $n\geq d$, then the Schur duality holds between the Hecke algebra $\cH_q(\Sigma_{d})$ and its corresponding $q$-Schur algebra $S_q(n,d)$.
{Moreover, a Schur functor exists.}
\end{cor}
\vskip .25cm

\subsubsection{Evseev-Kleshchev's Wreath Product Algebras}
Let $B$ be an arbitrary superalgebra, and let $R = 1\otimes 1, S = 0, \sigma: a\otimes b \mapsto b\otimes a$, and $\rho = 0$. 
So $A = B \wr \cH(d)$ is the super wreath product algebra $W_d^B = B^{\otimes d} \rtimes K \Sigma_d$ in \cite{EK17}.
{Note that $V_B^{\otimes d}$ admits an $A$-action by Proposition~\ref{rmk:Haction}(c)}.
Let $n\in\NN$, and let $V_B = \bigoplus_{i=1}^n B$ as a right $B$-supermodule in the usual way. 
If $n \geq d$, then the assumptions in Lemma~\ref{lem:QSchur} hold.
In this case, $U_i$'s are just those distinct summands in $V_B$.

Thus,  $S^A((B^{\otimes n})^{\otimes d})$ coincides with the Evseev-Kleshchev's generalized Schur algebra $S^B(n,d) := \End_{W^B_d}((B^{\otimes n})^{\otimes d})$, 
and hence we cover the double centralizer property result in \cite{EK17}.
\begin{cor}
If $n\geq d$, then the Schur duality holds between the super wreath product algebra $W^B_d$ and its corresponding Schur algebra $S^B(n,d)$.
{Moreover, a Schur functor exists.}
\end{cor}

\subsection{A Splitting Lemma, II}\label{sec:SchurA2}
{While in Section~\ref{sec:SchurA} we demonstrate an application of Theorem~\ref{thm:SchurEquiv} via a construction of the splitting over the quantum wreath product $B \wr \cH(d)$ from a splitting over the base algebra $B$, in this section we provide a new application of Theorem~\ref{thm:SchurEquiv} via a different construction of the splitting.
To be precise, given an algebra $A'$, we can obtain the desired splitting over $A'$ as long as we have a splitting over a larger algebra $A$ which contains $A'$ as a self-injective subalgebra.

\begin{lemma} \label{lem:split}
Let $A$ be an associative algebra, and let $A' \subseteq A$ be a self-injective subalgebra. 
Suppose that one has a splitting of a right $A$-module $T$ as in \eqref{E:splitW}. 
Then, one has a splitting $T \cong A' \oplus Q'$, as right $A'$-modules for some $Q'$.
\end{lemma}

\begin{proof}
Note that $A'$ is a right $A'$-submodule of $A$. By the self-injectivity of $A'$, the right $A'$-module $A'$ is a direct summand of $A$. The result follows from the fact that the right $A$-module $A$ is a direct summand of $T$.
\end{proof}
}

\subsubsection{Hu Algebras}

Following Section~\ref{ex:HHS},
let $B = \cH_q(\Sigma_m)$ be generated by $T_1, \dots, T_{m-1}$,
and let $\cA(m) = B\wr \cH(2)$ the Hu algebra generated by $B^{\otimes 2}$ and an Hecke-like generator $H$.
{Next, we will apply Lemma~\ref{lem:split} to the case $A' := \cA(m) \subseteq A := \cH_q(\Sigma_{2m})$.}

Let $V(n)$ be the $K$-span of $v_1, \dots, v_n$, and let $T = T(n, 2 m) := V(n)^{\otimes 2m}$. 
{
The Hecke algebra $\cH_q(\Sigma_{2m})$ (and hence the Hu algebra $\cA(m)$) acts on $T$ by \eqref{eq:vHi}.
Therefore, we obtain the following diagram regarding the relationship between the $q$-Schur algebra $S_q(n,2m)$ and the wreath Schur algebra $S^{\cA(m)}(T)$.
\eq
\Ba{{cccccccc}
S_q(n,2m)&\crr& T(n,2m) &\crl& \cH_{q}(\Sigma_{2m})
\\
\dsubseteq &&||&&\usubseteq
\\
S^{\cA(m)}(T) &\crr& V(n)^{\otimes 2m}& \crl& \cA(m)
}.
\endeq
}

\begin{cor}
If $n\geq 2m$, then the Schur duality holds between the Hu algebra $\cA(m)$ and its corresponding Schur algebra $S^{\cA(m)}(T(n,2m))$.
{Moreover, a Schur functor exists.}
\end{cor}
\proof
{Recall that  $A := \cH_q(\Sigma_{2m}) \supseteq A' := \cA(m) \cong B \wr \cH(2)$ where $B = \cH_q(\Sigma_m)$ is symmetric, $\sigma$ is the flip map, $\rho = 0$, and $R = z_m$ is invertible.
Thus, the Hu algebra is symmetric thanks to Lemma~\ref{prop:struc}(b), and hence $A'$ is a self-injective subalgebra of $A$.

When $n \geq 2 m$, one has a splitting $T \cong A \oplus Q$ as $A$-modules. 
Then, by Lemma \ref{lem:split}, on obtains another splitting $T \cong \cA(m) \oplus Q'$ as $\cA(m)$-modules. 
Therefore, we can apply Theorem~\ref{thm:SchurEquiv} and then obtain the desired Schur duality. 
The existence of a Schur functor follows from the discussion in Section \ref{sec:SF}.
\endproof
}

\subsection{Related Schur-Weyl Dualities}
In this section we consider Schur-Weyl dualities in which the corresponding base algebras are infinite-dimensional. In this case 
the notion of symmetric algebra is not defined.
In turns out that we can still construct a natural $A$-action to afford the splitting \eqref{E:splitW}, 
and in turn state the most natural double centralizer property which awaits for a uniform proof.

We also include a proof of Schur duality for the Ariki-Koike algebras using Theorem~\ref{thm:SchurEquiv}.
Note that our splitting lemma does not apply here.
\subsubsection{Affine Hecke Algebras}\label{sec:AHASchur}
Following Section~\ref{ex:AHA}, we have $B= K[X^{\pm1}]$ and $A = B\wr \cH(d)$ is isomorphic to the extended affine Hecke algebra $\cH_q^\ext(\Sigma_d)$ of type A. 
Let $n\in\NN$, and let $V_B = \widehat{V}(n)$ be the space spanned by $\{v_i ~|~ i\in \ZZ\}$ on which $X^{\pm1}$ acts by
\eq
v_i.X = v_{i+n},
\quad
v_i.X\inv = v_{i-n}. 
\endeq 
If $n \geq d$, then the assumptions in Lemma~\ref{lem:QSchur} hold.
In this case, $U_j := v_j.B (1\leq j \leq d)$ are indeed distinct summands.

One can extend the right $A$-action to the entire tensor space $T = V_B^{\otimes d}$ via \eqref{eq:vHi}.
Then $S^A(T)$ coincides with the affine $q$-Schur algebra of type A.
The corresponding Schur-Weyl duality is proved in \cite{CP94, Gr99} as below:
\begin{prop}[\cite{CP94, Gr99}]
If $n\geq d$, then the Schur duality holds between the affine Hecke algebra $\cH^\ext_q(\Sigma_{d})$ and its corresponding affine $q$-Schur algebra $S^\aff_q(n,d) := \End_{\cH^\ext_q(\Sigma_d)}(\widehat{V}(n)^{\otimes d})$.
\end{prop}
\vskip .25cm
\subsubsection{Ariki-Koike Algebras}
Consider  $F = \frac{{K[X]}}{\prod_{i=1}^m (X- q_i)}$, where $q_i$'s are parameters as $v_i$ in Section~\ref{ex:AK}. 
Let $n\in \NN$, and let $V_F=V(m,n)$ be the $F$-module with basis $\{v_1, \dots, v_{mn}\}$, on which the $F$-action is given by, for $1\leq j \leq m, 0 \leq k \leq n-1$,
\eq \label{eq:AKV}
v_{km+j}. X = \begin{cases}
v_{km+j+1} &\tif  j < m; 
\\
\sum_{i=1}^{m} (-1)^{i-1} e_i(q_1, q_2, \dots, q_m) v_{km+m-i}&\tif j=m,
\end{cases}
\endeq
where $e_i$'s are the elementary symmetric functions that appear in rewriting $\prod_{i=1}^m (X-q_i) =0$ into $X^m = \sum_{i=1}^{m} (-1)^{i-1} e_i(q_1, \dots, q_m) X^{m-i}$.

Next, we define an $A$-module $T$ such that $T = V_F^{\otimes d}$ as a vector space.
Each $H_i (1\leq i \leq d-1)$ acts by \eqref{eq:vHi},
and each generator $\={X^{(i)}} \in A$ acts on the right by $H_{i-1}\dots H_1 X^{(1)} H_1 \dots H_{i-1}$, where
$X^{(1)}$ acts on the first tensor factor by \eqref{eq:AKV}.

If $n \geq d$, then the splitting as in \eqref{E:splitW} exists, since the span of $\{v_\mu ~|~ 1\leq \mu_i \leq md\}$ is a direct summand of $T$ that is isomorphic to $A$ as a right $A$-module.
It is known that $A$ is symmetric (and hence self-injective) by \cite[Theorem~5.1]{MM98},
and thus, by Theorem~\ref{thm:SchurEquiv}, we cover the results on Schur duality in \cite{Gr97, BW18, BWW18}.
Moreover, when $m > 2$, $S^A(T)$ is a new Schur algebra for the Ariki-Koike algebra that has the double centralizer property. 
See \cite{SS99, JM00} for attempts to obtain such a Schur duality.

\begin{cor}
If $n\geq d$, then the Schur duality holds between the Ariki-Koike algebra $\cH_{q, q_1, \dots, q_m}(C_m \wr \Sigma_{d})$ and its corresponding Schur algebra $S^A(V(m,n)^{\otimes d})$.
\end{cor}
\vskip .25cm
\subsubsection{Affine Hecke Algebras of Type C}
Following Section \ref{ex:AHABC}, let {$F = \cH_q(\Sigma^\aff_2)$ be generated by $X,Y$ modulo the quadratic relations $X^2 = (\xi\inv - \xi)X + 1$ and $Y^2 = (\eta\inv -\eta)Y +1$ for parameters $\xi, \eta \in K^{\times}$. 
Let $V_F = \widetilde{V}(n)$ be the space spanned by $\{v_i ~|~ i\in \ZZ\}$, on which $F$ acts by
\eq\label{eq:AHACV}
v_i. X = \begin{cases}
v_{-i} &\tif j > 0; 
\\
v_{-i} + (\xi\inv - \xi) v_i &\tif j <0,
\end{cases}
\quad
v_i. Y = \begin{cases}
v_{2(n+1)-i} &\tif j < n+1; 
\\
v_{2(n+1)-i} + (\eta\inv - \eta) v_i &\tif j > n+1.
\end{cases}
\endeq
For the right $A$-module structure on $T = {V}_F^{\otimes d}$, $H_i$'s act by \eqref{eq:vHi}, 
 $\={X^{(i)}}, \={Y^{(d-i+1)}} \in A$ act on the right by $H_{i-1}\dots H_1 X^{(1)} H_1 \dots H_{i-1}$ and $H_{d-i} \dots H_{d-1}Y^{(d)}H_{d-1} \dots H_{d-i}$, respectively, where
$X^{(1)}$ and $Y^{(d)}$ act on the first and last tensor factors, respectively, by \eqref{eq:AHACV}.
}
Then $S^A(T)$ coincides with the affine $q$-Schur algebra of type C.
The corresponding Schur-Weyl duality, after a renormalization, is proved in  \cite{FLLLWW20} as below:
\begin{prop}[ \cite{FLLLWW20}]
If $n\geq d$, then the Schur duality holds between the affine Hecke algebra $\cH_{q,\xi,\eta}(C^\aff_{d})$ and its corresponding affine $q$-Schur algebra $\widehat{S}^{\mathfrak{c}}_{q,\xi,\eta}(n,d) := \End_{\cH_{q,\xi,\eta}(C^\aff_{d})}(\widetilde{V}(n)^{\otimes d})$.
\end{prop}
\vskip .25cm
\subsubsection{Degenerate Hecke Algebras of Type A}
Following Section~\ref{ex:dAHA}, we have $B= K[X]$ and $A = B\wr \cH(d)$ is isomorphic to the degenerate affine Hecke algebra $\cH^\Deg(\Sigma_d)$ of type A. 
Let $n\in\NN$, and let $V_B = \={V}(n)$ be the space spanned by $\{v_i ~|~ i\in \ZZ_{\geq 1} \}$ on which $X$ acts by
$v_i.X = v_{i+n}$.
If $n \geq d$, then the assumptions in Lemma~\ref{lem:QSchur} hold.
In this case, $U_j := v_j.B (1\leq j \leq d)$ are indeed distinct summands.
One can extend the right $A$-action to the entire tensor space $T = V_B^{\otimes d}$ via \eqref{eq:vHi}.
Then $S^A(T)$ seems to be  a new $q$-Schur algebra for the degenerate affine Hecke algebra.

See \cite{BK08} for a related cyclotomic version.
Note that our approach does not recover the Schur duality functor as in \cite{AS98} since their algebra $\cH^\Deg(\Sigma_d)$ acts on a larger module $M \otimes V^{\otimes d}$.
\subsection{Connections to Quantum Groups}
In this article, we only discuss a Schur duality between a quantum wreath product and its corresponding Schur algebra;
while in literature, there is usually a third object that surjects onto the Schur algebra.
For example, it is the Drinfeld-Jimbo quantum group that surjects onto the $q$-Schur algebras of type A.

A standard procedure due to Beilinson-Lusztig-MacPherson can be used to produce these higher level algebras from the given families of Schur algebras.
See \cite{BLM90} for constructing the quantum group $U_q(\mathfrak{gl_n})$ from $\cH_q(\Sigma_d)$, 
and \cite{BKLW18, LL21} for a realization of equal or multi-parameter $\imath$quantum groups $\imath U(\fgl_n)$ from Hecke algebra of type B/C. 
Furthermore, \cite{FLLLW23} explains how to obtain the affine $\imath$quantum groups of type C using affine Hecke algebra of type C.
It is still open whether one can generalize this approach to obtain the quantum group analogs for other examples in Section~\ref{sec:examples}, e.g., the Ariki-Koike algebras, the Hu algebras, or the degenerate affine Hecke algebras. 
\appendix

\section{Examples of Quantum Wreath Product Algebras} \label{sec:app}
\subsection{Affine Hecke Algebras}\label{sec:AHA}
We first consider the extended affine Hecke algebra $\cH^\ext_q(\Sigma_d)$ for $\GL_d$. 
Recall that $P = \sum_{i=1}^d \ZZ \epsilon_i$ is the weight lattice for $\GL_d$, and $\alpha_i := \epsilon_i - \epsilon_{i+1}$ is our choice of positive simple roots.
The algebra $\cH^\ext_q(\Sigma_d)$ is generated by $T_1, \dots, T_{d-1}$ and $Y^\lambda (\lambda \in P)$ subject to the Bernstein-Lusztig relation, for $\lambda \in P, 1\leq i \leq d-1$: 
\eq\label{eq:BLrel}
T_i Y^\lambda=Y^{s_i(\lambda)} T_i + (q-1) \frac{Y^\lambda - Y^{s_i(\lambda)}}{1- Y^{-\alpha_i}},
\endeq 
as well as the relations below:
\eq
\begin{split}
Y^\lambda Y^\mu = Y^{\lambda + \mu}, \quad
&
T_i T_{i+1} T_i = T_{i+1}T_i T_{i+1},
\quad
T^2_i = (q-1) T_i + q,
\\
&T_i T_j = T_j T_i \quad (|i-j|>1).
\end{split}
\endeq
We remark that the non-extended affine Hecke algebra $\cH^\aff_q(\Sigma) := \cH_q(\Sigma^\aff_d)$ of type A is the Hecke algebra  for the Coxeter group $\Sigma^\aff_d$.
Furthermore, $\cH^\aff_q(\Sigma_d)$ can be embedded into $\cH^\ext_q(\Sigma_d)$ via
\eq
T_0 \mapsto Y^{\epsilon_1-\epsilon_d} T_{d-1 \to 1 \to d-1},
\quad
T_i \mapsto T_i (i>0).
\endeq

\subsection{Degenerate Affine Hecke Algebras}\label{sec:dAHA}
The degenerate affine Hecke algebra of type A is the algebra $\cH^\Deg(d)$ generated by $s_1, \dots, s_{d-1}$ and the polynomial algebra $K[x_1, \dots, x_d]$ subject to the relations below,
for $1\leq k \leq d-2, 1\leq i \leq d-1, |j-i|\geq 2$:
\eq
s_k s_{k+ 1} s_k = s_{k+ 1} s_k s_{k + 1},  \quad s_i s_j = s_j s_i,
\endeq 
as well as the following cross relations:
\eq\label{eq:sxrel}
s_i x_i = x_{i+1} s_i  -1, 
\quad
s_i x_{i+1} = x_{i} s_i  + 1,
\quad s_i x_j = x_j s_i \quad (j \neq i, i+1).
\endeq

\subsection{Nil Hecke Rings}\label{sec:NHR}
Kostant-Kumar's nil Hecke ring \cite{KK86} (or, affine nil Hecke algebra) of type A is the algebra $N\cH(d)$ generated by $\partial_1, \dots, \partial_{d-1}$ and the polynomial algebra $K[x_1, \dots, x_d]$ subject to the relations below,
for $1\leq k \leq d-2, 1\leq i \leq d-1, |j-i|\geq 2$:
\eq
\begin{split}
&\partial_i^2=0, 
\quad
\partial_k \partial_{k+ 1} \partial_k = \partial_{k+ 1} \partial_k \partial_{k + 1},  \quad \partial_i \partial_j = \partial_j \partial_i,
\\
&\partial_i x_i = x_{i+1} \partial_i  -1, 
\quad
\partial_i x_{i+1} = x_{i} \partial_i  + 1,
\quad \partial_i x_j = x_j \partial_i \quad (j \neq i, i+1).
\end{split}
\endeq

\subsection{Ariki-Koike Algebras}\label{sec:AK}
The Ariki-Koike algebra (i.e., cyclotomic Hecke algebra), denoted by $\cH_{m,d}$, is a deformation of  $G(m,1,d) = C_m \wr \Sigma_d$ of parameters $(q_1, \ldots , q_m, q)$ over a Laurent polynomial $F = K[q_1^{\pm1} , \ldots, q_m^{\pm1}, q^{\pm1}]$
generated by $X, T_1, \ldots T_{d-1}$ subject to  the type A braid relations for $T_1, \dots, T_{d-1}$, as well as the  the following relations, for $1\leq i < d$:
\eq
\prod_{1\leq j \leq m}(X- q_j)
= 0 =
(T_i+1)(T_i-q).
\endeq
Note that the specialization of  $\cH_{m,d}$ at $q_i = \xi^i, q=1$ is the group algebra $KG(m,1,d)$ for $\xi$ a primitive $m$th root of unity.

\subsection{Rosso-Savage Algebras}\label{def:RS}
\subsubsection{Affine Wreath Product Algebra}
Wan-Wang's wreath Hecke algebras is a special case of  Savage's affine wreath product algebra $\cA_d(B)$.
Following \cite{Sa20}, here $B$ is an $\NN$-graded Frobenius superalgebra with homogeneous parity-preserving linear trace map $\Tr:B\to K$.
Denote its Nakayama automorphism by $\psi:B\to B$, i.e., 
\eq
\Tr(xy) =(-1)^{\={x}\={y}} \Tr(y \psi(x)) \quad x,y \in B.
\endeq
Assume that $B$ has a basis $I$ consisting of homogeneous elements, and let
 $\{x^\vee~|~x \in I\}$ be its dual basis in the sense that $\Tr(x^\vee y) = \delta_{x,y}$ for all $x, y \in I$.

For $x\in B$, denote by $x^{(i)}$ the element in $B^{\otimes d}$ such that all tensor factors are identities except that the $i$th factor is $x$. 
Next, for $b = b_1 \otimes \dots \otimes b_d \in B^{\otimes d}$, let $\psi_i(b) = \psi(b_i)^{(i)} \in B^{\otimes d}$.
Furthermore, define 
$t_{i,j} := \sum_{x \in I} x^{(i)} (x^\vee)^{(j)} \in B^{\otimes d}$ for all $i,j$.

The algebra $\cA_d(B)$ is generated by $b = b_1 \otimes  \dots \otimes b_d \in B^{\otimes d}, L_1,\dots, L_d, T_1, \dots, T_{d-1}$ subject to the relations below, for $1\leq i \leq d-1, 1\leq j,k \leq d$:
\eq
\begin{split}
&L_k L_j = L_j L_k, \quad L_i \psi_i(b) = b L_i,
\\
&T_i^2 = 1, \quad  T_i T_{i+ 1} T_i = T_{i+ 1} T_i T_{i + 1},  \quad T_i T_j = T_j T_i \quad (|j-i|>1),
\\
&T_i b = s_i(b) T_i, \quad T_iL_i = L_{i+1}T_i - t_{i,i+1}.
\end{split}
\endeq

\subsubsection{Quantum Wreath Algebra (Frobenius Hecke Algebra)}
Let $z\in K$.
Following \cite{RS20}, the Frobenius Hecke algebra (or, quantum wreath algebra) $H_d(B,z)$ is generated by 
$b = b_1 \otimes  \dots \otimes b_d \in B^{\otimes d}, T_1, \dots, T_{d-1}$ subject to the relations below, for $1\leq i, j \leq d-1$:
\eq
\begin{split}
&T_i^2 = z t_{i,i+1} T_i + 1, 
\\
&T_i T_{i+ 1} T_i = T_{i+ 1} T_i T_{i + 1},  \quad T_i T_j = T_j T_i \quad (|j-i|>1),
\\
&T_i b = s_i(b) T_i.
\end{split}
\endeq
\subsubsection{Quantum Affine Wreath Algebra (Affine Frobenius Hecke Algebra)}
Let $z\in K$.
Following \cite{RS20}, the affine Frobenius Hecke algebra (or, quantum affine wreath algebra) $H^\aff_d(B,z)$ 
is generated by $b = b_1 \otimes  \dots \otimes b_d \in B^{\otimes d}, X^{\pm}_1,\dots, X^{\pm}_d, T_1, \dots, T_{d-1}$ subject to the relations below, for $1\leq i \leq d-1, 1\leq j,k \leq d$:
\eq
\begin{split}
&X^{\pm}_k X^{\pm}_j = X^{\pm}_j X^{\pm}_k, \quad X_i b = b X_i,
\\
&T_i^2 = 1, \quad  T_i T_{i+ 1} T_i = T_{i+ 1} T_i T_{i + 1},  \quad T_i T_j = T_j T_i \quad (|j-i|>1),
\\
&T_i b = s_i(b) T_i, \quad T_i X_i T_i = X_{i+1}, \quad T_i X_j = X_j T_i \quad (|j-i|>1).
\end{split}
\endeq

\subsection{Affine Zigzag Algebras} \label{def:KM}
Let $\Gamma$ be a connected Dynkin diagram of finite type ADE, and let $\={\Gamma}$ be its corresponding double quiver with vertex set $I$ and edge set $E = \{ a_{i,j} ~|~ i \neq j \in I\}$.
Let $Z_K = Z_K(\Gamma)$ be the zigzag algebra over $K$ be the quotient of the path algebra $K \={\Gamma}$ subject to the following relations:
\begin{enumerate}
\item Any path of length $\geq 3$ is zero;
\item Any path of length $2$ is zero, unless it is a cycle;
\item All cycles of length 2 at the same node are equal.
\end{enumerate}
In other words, $Z_K$ has a basis $\{e_i, a_{i,j}, c_i ~|~ j\neq i \in I\}$ such that $e_i$'s are the length-0 paths at $i$, $c_i$'s are the length-2 cycles at $i$ (which are identified as long as they have the same base vertex).
The identity in $Z_K$ is $1_Z = \sum_{i\in I} e_i$.
We can paraphrase the definition (see \cite{KM19}) of the affine zigzag algebra $Z^\aff_d(\Gamma)$ as follows. Let $Z^\aff_d(\Gamma)$ be generated by $Z_K^{\otimes d}, z_1, \dots z_d, T_1, \dots, T_{d-1}$ subject to the relations below, 
for $1\leq i \leq d-1, 1\leq j,k \leq d, \gamma \in Z_K^{\otimes d}$:
\eq
\begin{split}
&z_k z_j = z_j z_k, \quad z_i \gamma = \gamma z_i,
\\
&T_i^2 = 1, \quad  T_i T_{i+ 1} T_i = T_{i+ 1} T_i T_{i + 1},  \quad T_i T_j = T_j T_i \quad (|j-i|>1),
\\
&T_i \gamma = s_i(\gamma) T_i, 
\\ 
&T_i z_j e_\lambda = z_{s_i(j)} e_{s_i(\lambda)}T_i + \begin{cases}
(\delta_{i,j} - \delta_{i+1,j})(c_j + c_{j+1})e_\lambda &\tif \lambda_i = \lambda_{i+1};
\\
(\delta_{i,j} - \delta_{i+1,j})(a_{\lambda_{i+1}, \lambda_{i}} \otimes a_{\lambda_{i}, \lambda_{i+1}})_{{i}}e_\lambda &\tif a_{\lambda_i, \lambda_{i+1}}\in E;
\\
0 &\textup{otherwise,}
\end{cases}
\end{split}
\endeq
where $e_\lambda = e_{\lambda_1} \otimes \dots \otimes e_{\lambda_d}$ for $\lambda = (\lambda_i)_i \in I^d$ 
and $\Sigma_d$ acts on $ Z_K^{\otimes d}$ by place permutations.
\section{Calculations for the Basis Theorem} \label{sec:loopproof}

\subsection{Proof of Proposition~\ref{prop:nec}} 
The proof will entail a term-by-term comparison via the basis $\{(b_1\otimes \dots \otimes b_d) H_w\}$. 
{For the proofs within this appendix, it is enough to check local relations in either $B\otimes B$ or $B^{\otimes 3}$. The general proofs follow by using the embedding into $B^{\otimes d}$ for an arbitrary $d$.}

Equation \eqref{def:wr1} follows from the fact that $H(1\otimes 1) = \sigma(1\otimes 1) H + \rho(1\otimes 1) = (1\otimes 1) H$.
Equations \eqref{def:wr2}--\eqref{def:TTT} follow from the fact that $H(ab) = \sigma(ab) H + \rho(ab)$, while
\eq
(Ha)b = \sigma(a) H b+\rho(a)b = \sigma(a)\sigma(b)H+\sigma(a)\rho(b) + \rho(a)b;
\endeq
and the fact that $(H^2)H = (SH+R)H = (S^2+ R)H +SR$, while
\eq
H(H^2)
= (\sigma(S)H + \rho(S))H + \sigma(R)H + \rho(R)
= (\sigma(S)S+\rho(S)+\sigma(R))H +\sigma(S)R {+ \rho(R)}.
\endeq
Equation \eqref{def:qu1} follows from the fact that $H^2a = (SH+R)a = S(\sigma(a)H + \rho(a)) + Ra$; while $H(Ha)$ is equal to
\eq
\begin{split}
&H(\sigma(a)H + \rho(a)) = (H\sigma(a))H + H\rho(a) 
= (\sigma^2(a)H+\rho\sigma(a))H + \sigma\rho(a)H + \rho^2(a)
\\
&~= (\sigma^2(a){S} + \rho\sigma(a) + \sigma\rho(a))H + \rho^2(a)+\sigma^2(a){R}.
\end{split}
\endeq
{Assume that $d\geq 3$ from now on.}
Equations \eqref{def:br1}--\eqref{def:br3} follow from combining 
\eq\label{eq:br1-3}
\begin{split}
H_1H_2H_1a 
&= \sigma_1\sigma_2\sigma_1(a)H_1H_2H_1 +  \rho_1\sigma_2\sigma_1(a)H_2H_1
+ \sigma_1\rho_2\sigma_1(a)(S_{1}H_{1}+R_{1})+  \rho_1\rho_2\sigma_1(a)H_1
\\
&+ \sigma_1\sigma_2\rho_1(a)H_1H_2 +  \rho_1\sigma_2\rho_1(a)H_2
+ \sigma_1\rho_2\rho_1(a)H_1+  \rho_1\rho_2\rho_1(a),
\end{split}
\endeq
with a similar formula for $H_2H_1H_2a$ using symmetry. 
For \eqref{def:br4}-\eqref{def:br5}, one uses $H_1(H_1H_2H_1) = S_1 H_1H_2 H_1 + R_1 H_2 H_1$; while
\eq
\begin{split}
H_2H_1H_2^2 &=  \sigma_2\sigma_1(S_2)H_2H_1H_2 +  \sigma_2\sigma_1(R_2)H_2H_1
+ \rho_2\sigma_1(S_2)H_1H_2 +  \rho_2\sigma_1(R_2)H_1
\\
&+ \sigma_2\rho_1(S_2)S_2H_2 +  \sigma_2\rho_1(S_2)R_2
+ \rho_2\rho_1(S_2)H_2+  \rho_2\rho_1(R_2).
\end{split}
\endeq

\subsection{Proof of Proposition~\ref{prop:loop}} In this section, we use the postfix notation for maps of the form $f_a$ and $T_i$ with $a\in B^{\otimes d}$. It is the same postfix notation as what we used in 
 \eqref{eq:postfix}. Note that in the subscripts of the $f_a$'s, we still use the regular prefix composition and evaluation for $\sigma_i$ and $\rho_i$.

In the following, we will use the fact that $f_a + f_b = f_{a+b}$ for all $a, b\in B^{\otimes d}$ (see Proposition~\ref{prop:Klinearity}) without referring to it.
For simplicity, we abbreviate $1^{\otimes d}$ by $1$.
When we refer to Condition \hyperlink{eq:Rl}{$R[\ell]$}, we will use the following equivalent form:
\eq
	(1\otimes w) \cdot f_{a} = (1 \otimes w s_i) \cdot ( f_{\sigma_i(a)} T_{i} + f_{\rho_i(a)})
	\quad
	(\textup{for all }w \textup{ with }
	\ell(w s_i) \lt \ell(w) \leq \ell).
\endeq
Finally, it suffices to prove each part of Proposition~\ref{prop:loop} by considering the {postfix evaluation at }  $1 \otimes w$ for every $w \in \Sigma_d$ with $\ell(w) \leq \ell$.

\begin{proof}[Proof of Part (W)]
First, consider the case when $\ell(w) <  \ell(ws_i)$. We have
\eq
  (1 \otimes w) \cdot T_if_a 
  =   (1 \otimes ws_i) \cdot f_a= (1 \otimes w) \cdot(f_{\sigma_i(a)} T_i + f_{\rho_i(a)}). 
\quad \text{ by \hyperlink{eq:Rl}{$R[\ell + 1]$}}
\endeq
For the other case when $\ell(w) >  \ell(ws_i)$, note that $(1 \otimes w)\cdot  T_i f_a =  (1 \otimes ws_i) \cdot T_i^2f_a$.
As elements in $\Hom_{B^{\otimes d}}(V^{\ell-1}, V)
$, 
\begin{allowdisplaybreaks}
\begin{align}
T_i^2f_a& 
	= f_{S_i}T_i f_a +  f_{R_i}f_a  \ \ \ \text{ by \hyperlink{eq:Ql}{$Q[\ell - 1]$}} \\
\notag 
= & f_{S_i\sigma_i(a)} T_i 
	+ f_{ S_i\rho_i(a)} 
	+ f_{R_i a}  \ \ \ \text{ by \hyperlink{eq:Wl}{$W[\ell - 1]$}, \hyperlink{eq:Ml}{$M[\ell - 1]$}} \\
\notag 
= &	(	f_{\sigma_i^2(a)} f_{S_i}
		+ f_{\rho_i\sigma_i(a)} 
		+ f_{\sigma_i\rho_i(a)}
	)T_i 
	+ f_{\sigma_i^2(a)} f_{R_i} 
	+ f_{\rho_i^2(a)}  \ \ \ \text{ by \eqref{def:qu1}, \hyperlink{eq:Ml}{$M[\ell - 1]$}} \\
\notag
= & (f_{\sigma_i^2(a)} T_i 
	+ f_{\rho_i\sigma_i(a)} 
	+ f_{\sigma_i\rho_i(a)}
	)T_i 
	+ f_{\rho_i^2(a)}  \ \ \ \text{ by \hyperlink{eq:Ql}{$Q[\ell - 1]$}} \\
\notag
= & T_i f_{\sigma_i(a)} T_i 
	+ T_i f_{\rho_i(a)}   \ \ \ \text{ by \hyperlink{eq:Wl}{$W[\ell - 1]$}}.
\end{align}	
\end{allowdisplaybreaks}
Therefore, $(1 \otimes w)\cdot  T_i f_a 
=  (1 \otimes ws_i) \cdot  (T_i f_{\sigma_i(a)} T_i + T_i f_{\rho_i(a)}) =  (1 \otimes w) \cdot (f_{\sigma_i(a)} T_i  + f_{\rho_i(a)})$.
\end{proof}
\begin{proof}[Proof of Part (M)]
The case when $w=1$ follows immediately from the construction.  
We may now assume that $w > 1$. Then, there exists $s_i$ such that $\ell(w) >  \ell(ws_i)$.
Note first $(1 \otimes w)\cdot f_{a} f_b   =  (1 \otimes ws_i) \cdot T_i f_a f_b$ and
$(1 \otimes w)\cdot f_{ab} = (1 \otimes ws_i) \cdot T_i f_{ab}$.
We are done since the equalities below hold in $\Hom_{B^{\otimes d}}(V^{\ell-1}, V)
$:
\begin{allowdisplaybreaks}
\begin{align}
T_i f_a f_b& 	
	= f_{\sigma(a)} T_i f_b + f_{\rho_i(a)} f_b \ \ \ \text{ by \hyperlink{eq:Wl}{$W[\ell - 1]$}} \\
\notag &
	=  f_{\sigma_i(a) \sigma_i(b)}T_i + f_{\sigma_i(a) \rho_i(b)} + f_{\rho_i(a)b} \ \ \ \text{ by \hyperlink{eq:Wl}{$W[\ell - 1]$}, \hyperlink{eq:Ml}{$M[\ell - 1]$}} \\
\notag &
	=  f_{\sigma_i(ab)}T_i + f_{\rho_i(ab)} \ \ \ \text{ by \eqref{def:wr2}} \\
\notag &
	= T_i f_{ab}. \ \ \ \text{ by \hyperlink{eq:Wl}{$W[\ell - 1]$}}
\end{align}
\end{allowdisplaybreaks}
\end{proof}
\begin{proof}[Proof of Part (Q)]
The case when $\ell(w) < \ell(w s_i)$ follows immediately from construction.
For $\ell(w) > \ell(w s_i)$, it suffices to show that $T_i^3 = T_i  f_{S_i} T_i + T_if_{R_i}$ in $\Hom_{B^{\otimes d}}(V^{\ell-1}, V)
$.
Indeed,
\begin{allowdisplaybreaks}
\begin{align}
T_i^3& 
	=  f_{S_i} T_i^2  + f_{R_i} T_i  \ \ \ \text{ by \hyperlink{eq:Ql}{$Q[\ell - 1]$}} \\
\notag&
	=  (f_{S_i^2} + f_{R_i})T_i + f_{S_i R_i}  \ \ \ \text{ by \hyperlink{eq:Ql}{$Q[\ell - 1]$}, \hyperlink{eq:Ml}{$M[\ell - 1]$}} \\
\notag&
	=  ( f_{\sigma_i(S_i)}f_{S_i} + f_{\rho_i(S_i)} + f_{\sigma_i(R_i)} 
		)T_i 
	 + f_{\rho_i(R_i)} +   f_{\sigma_i(S_i)}f_{R_i}  \ \ \ \text{ by \eqref{def:TTT}, \hyperlink{eq:Ml}{$M[\ell - 1]$}} \\
\notag&
	=  f_{\sigma_i(S_i)} T_i^2 +  f_{\rho_i(S_i)} T_i +  f_{\sigma_i(R_i)}T_i + f_{\rho_i(R_i)}  \ \ \ \text{ by \hyperlink{eq:Ql}{$Q[\ell - 1]$}} \\
\notag&
	= T_i  f_{S_i}  T_i + T_i  f_{R_i}.  \ \ \ \text{ by \hyperlink{eq:Wl}{$W[\ell - 1]$}} 
\end{align}
\end{allowdisplaybreaks}
\end{proof}
\begin{proof}[Proof of Part (B2)]
We first consider $w$ with $\ell(w) \lt \ell(w s_i)$ and $\ell(w) \lt \ell(w s_j)$. Then
\eq
 ( 1 \otimes w) \cdot T_i  T_j 
=  1\otimes w s_i s_j 
=  1\otimes w s_j s_i  
=  ( 1 \otimes w) \cdot T_j  T_i.
\endeq
By symmetry in $i$ and $j$, it remains to prove (B2) for the case $\ell(w) > \ell(w s_i)$. Since $1 \otimes w = (1 \otimes w s_i) \cdot T_i$, it suffices to prove in $\Hom_{B^{\otimes d}}(V^{\ell-1}, V)
$ that $T_i T_j T_i = T_i^2 T_j$.

\begin{allowdisplaybreaks}
\begin{align}
T_i T_j T_i = & T_j T_i^2   \ \ \ \text{ by \hyperlink{eq:B2l}{$B_2[\ell -  1]$}} \\
\notag = & T_j f_{S_i}T_i  + T_jf_{R_i}  \ \ \ \text{ by \hyperlink{eq:Ql}{$Q[\ell]$}} \\
\notag = & f_{\sigma_j(S_i)} T_j T_i +  f_{\rho_j(S_i)} T_i +  f_{\sigma_j(R_i)} T_j + f_{\rho_j(R_i)}  \ \ \ \text{ by \hyperlink{eq:Wl}{$W[\ell - 1]$}} \\
\notag = & f_{S_i} T_j T_i +  f_{R_i} T_j  \ \ \ \text{ by \eqref{def:wr1} and ${|i-j|>1}$} \\
\notag = & f_{S_i} T_i T_j + f_{R_i} T_j   \ \ \ \text{ by \hyperlink{eq:B2l}{$B_2[\ell -  1]$}} \\
\notag = & T_i^2 T_j.  \ \ \ \text{ by \hyperlink{eq:Ql}{$Q[\ell - 1]$}} 
\end{align}
\end{allowdisplaybreaks}

\end{proof}
\begin{proof}[Proof of Part (B3)]
The proof for the case with $\ell(w) \lt \ell(ws_i)$ and $\ell(w) \lt \ell(w s_j)$ is almost identical to the proof of Part (B2), and hence is omitted.
By symmetry in $i$ and $j$, it suffices to prove (B3) for the case $\ell(w) > \ell(w s_i)$. Since $1 \otimes w = (1 \otimes w s_i) \cdot T_i$, it suffices to check $T_i T_j  T_i  T_j   = T_i^2  T_j  T_i$ in $\Hom_{B^{\otimes d}}(V^{\ell-1}, V)
$.
\begin{allowdisplaybreaks}
\begin{align}
T_i T_j  T_i  T_j  
= & T_j   T_i  T_j^2   \ \ \ \text{ by \hyperlink{eq:B3l}{$B_3[\ell - 1]$}} \\
\notag
= & T_j T_i f_{S_j} T_j +   T_j  T_i f_{R_j}  \ \ \ \text{ by \hyperlink{eq:Ql}{$Q[\ell + 1]$}} \\
\notag
= & T_j    f_{\sigma_i(S_j)} T_i T_j 
	+ T_j  f_{\rho_i(S_j)}  T_j 
	+ T_j  f_{\sigma_i(R_j)}  T_i 
	+ T_j f_{\rho_i(R_j)}   \ \ \ \text{ by \hyperlink{eq:Wl}{$W[\ell]$}} \\
\notag
= & f_{\sigma_j\sigma_i(S_j)}  T_j  T_i  T_j  
	+ f_{\rho_j\sigma_i(S_j)} T_i  T_j  
	+ f_{\sigma_j\rho_i(S_j)} T_j^2   
	+ f_{\rho_j\rho_i(S_j)} T_j   \\
\notag&~
	+  f_{\sigma_j\sigma_i(R_j)} T_j T_i  
	+   f_{\rho_j\sigma_i(R_j)} T_i
	+   f_{\sigma_j\rho_i(R_j)} T_j
	+ f_{\rho_j\rho_i(R_j)}  \ \ \ \text{ by \hyperlink{eq:Wl}{$W[\ell - 1]$}} \\
\notag
= &   f_{\sigma_j\sigma_i(S_j)} T_j  T_i  T_j
	+  f_{\rho_j\sigma_i(S_j)} T_i  T_j
	+  f_{\sigma_j\sigma_i(R_j)} T_j  T_i
	+ (f_{\sigma_j\rho_i(S_j) S_j} 
	+ f_{\rho_j\rho_i(S_j)} 
	+ f_{\sigma_j\rho_i(R_j)}
		)T_j \\
\notag&~
	+ f_{\rho_j\sigma_i(R_j)} T_i
	+ f_{\sigma_j\rho_i(S_j) R_j} 
	+ f_{\rho_j\rho_i(R_j)}  \ \ \ \text{ by \hyperlink{eq:Ql}{$Q[\ell - 1]$}, \hyperlink{eq:Ml}{$M[\ell - 1]$}} \\
\notag
= &   f_{\sigma_j\sigma_i(S_j)} T_i  T_j  T_i
	+  f_{\rho_j\sigma_i(S_j)} T_i  T_j
	+   f_{\sigma_j\sigma_i(R_j)} T_j  T_i
	+   (f_{\sigma_j\rho_i(S_j) S_j} 
	+ f_{\rho_j\rho_i(S_j)} 
	+ f_{\sigma_j\rho_i(R_j)}
	)T_j \\
\notag &
+   f_{\rho_j\sigma_i(R_j)} T_i
	+ f_{\sigma_j\rho_i(S_j) R_j} 
	+ f_{\rho_j\rho_i(R_j)}  \ \ \ \text{ by \hyperlink{eq:B3l}{$B_3[\ell - 1]$}} \\
\notag = & f_{S_i}T_i  T_j  T_i +  f_{R_i}T_j  T_i  \ \ \ \text{ by \eqref{def:br4} and \eqref{def:br5}} \\
\notag = & T_i ^2 T_j  T_i.  \ \ \ \text{ by \hyperlink{eq:Ql}{$Q[\ell - 1]$}} 
\end{align}
\end{allowdisplaybreaks}
\end{proof}
\begin{proof}[Proof of Part (R)]
{Every pair of reduced expressions of $w$ differs by a chain of braid relations. Therefore, it suffices to prove that, for two reduced expressions $r, r'$ of $w$ such that $r$ and $r'$ differ by a single braid relation,
\eq \label{eq:rex}
(1 \otimes w s_i) \cdot (f_{\sigma_i(a)} T_i + f_{\rho_i(a)})
=
(1 \otimes w s_j) \cdot (f_{\sigma_j(a)} T_j + f_{\rho_j(a)}),
\endeq
where $s_i$ (resp. $s_j$) is the last simple generator of $r$ (resp. $r'$).
Below, the left-hand side and the right-hand side of \eqref{eq:rex} are denoted by LHS and RHS, respectively.

Equation \eqref{eq:rex} holds trivially if $s_i = s_j$. Now, assume that $s_i \neq s_j$, hence there exists $x \in \Sigma_d$ such that $w = x \beta_{i, j} = x \beta_{j, i}$ with $\ell(w) = \ell(x) + \ell(\beta_{i, j})$, where
\eq
\beta_{i,j} = \begin{cases}
s_j s_i s_j &\tif |i-j| =1;
\\
s_i s_j &\tif |i-j| >1.
\end{cases} 
\endeq

Note that in following, we must avoid using \hyperlink{eq:Wl}{$W[\ell - 1]$}, since otherwise a circular argument will occur in the Grand Loop Argument.}
Consider first the case that $\ell(\beta_{i, j}) = 2$.
\eq
\begin{split}
\textup{LHS} &~ = (1 \otimes x s_j) \cdot ({f_{\sigma_i(a)} T_i} + f_{\rho_i(a)})   \\
&~
= ( 1 \otimes x) \cdot ({f_{\sigma_j\sigma_i(a)} T_j T_i 
	+ f_{\rho_j\sigma_i(a)} T_i
	+ f_{\sigma_j\rho_i(a)} T_j}
	+ f_{\rho_j\rho_i(a)})  \ \ \ \text{ by \hyperlink{eq:Rl}{$R[\ell - 1]$}} \\
&~
= ( 1 \otimes x)\cdot ({f_{\sigma_i\sigma_j(a)} T_i T_j  
	+ f_{\sigma_i\rho_j(a)} T_i
	+ f_{\rho_i\sigma_j(a)} T_j}
	+ f_{\rho_i\rho_j(a)})   \ \ \ \text{ by \hyperlink{eq:B2l}{$B_2[\ell - 2]$}} \\
&~
= ( 1 \otimes xs_i) \cdot (T_j f_{\sigma_j(a)} + f_{\rho_j(a)}) 
=  \textup{RHS}.
\end{split}
\endeq

Finally, consider the case that $\ell(\beta_{i, j}) = 3$.
Note that
\eq\label{eq:LHSR3}
\begin{split}
\textup{LHS} = & (1 \otimes x s_i s_j) \cdot (f_{\sigma_i(a)} T_i + f_{\rho_i(a)})   \\
= & ( 1 \otimes x s_i) \cdot ({f_{\sigma_j\sigma_i(a)} T_j T_i 
	+ f_{\rho_j\sigma_i(a)} T_i
	+ f_{\sigma_j\rho_i(a)} T_j}
	+ f_{\rho_j\rho_i(a)})  \ \ \ \text{ by \hyperlink{eq:Rl}{$R[\ell - 1]$}} \\
= & ( 1 \otimes x)\cdot (
	{f_{\sigma_i\sigma_j\sigma_i(a)} T_i  T_j  T_i   
	+ f_{\rho_i\sigma_j\sigma_i(a)} T_j  T_i   
	+  f_{\sigma_i\rho_j\sigma_i(a)}  T_i^2
	+ f_{\sigma_i\sigma_j\rho_i(a)} T_i  T_j } \\
&~~
	+ {(f_{\rho_i\rho_j\sigma_i(a)} + f_{\sigma_i\rho_j\rho_i(a)}) T_i  
	+ f_{\rho_i\sigma_j\rho_i(a)}  T_j } 
	+ f_{\rho_i\rho_j\rho_i(a)}),  \ \ \ \text{ by \hyperlink{eq:Rl}{$R[\ell - 2]$}} \\
\end{split}
\endeq
and similarly,
\eq\label{eq:RHSR3}
\begin{split}
\textup{RHS} = & (1 \otimes x s_j s_i) \cdot (f_{\sigma_j(a)} T_j + f_{\rho_j(a)})   \\
= & (1 \otimes x s_j) \cdot (f_{\sigma_i\sigma_j(a)} T_i T_j 
	+ f_{\rho_i\sigma_j(a)} T_j
	+ f_{\sigma_i\rho_j(a)} T_i
	+ f_{\rho_i\rho_j(a)})  \ \ \ \text{ by \hyperlink{eq:Rl}{$R[\ell - 1]$}} \\
= & ( 1 \otimes x)\cdot (
	{f_{\sigma_j\sigma_i\sigma_j(a)} T_j  T_i  T_j   
	+ f_{\rho_j\sigma_i\sigma_j(a)} T_i  T_j   
	+  f_{\sigma_j\rho_i\sigma_j(a)}  T_j^2
	+ f_{\sigma_j\sigma_i\rho_j(a)} T_j  T_i } \\
&~~
	+{ (f_{\rho_j\rho_i\sigma_j(a)} + f_{\sigma_j\rho_i\rho_j(a)}) T_j  
	+ f_{\rho_j\sigma_i\rho_j(a)}  T_i  
	+ f_{\rho_j\rho_i\rho_i(a)}}).  \ \ \ \text{ by \hyperlink{eq:Rl}{$R[\ell - 2]$}} \\
\end{split}
\endeq
The proposition follows by showing the following identity in $\Hom_{B^{\otimes d}}(V^{\ell - 3}, V)
$:
\begin{allowdisplaybreaks}
\begin{align}
&	f_{\sigma_i\sigma_j\sigma_i(a)} T_i  T_j  T_i   
	+ f_{\rho_i\sigma_j\sigma_i(a)} T_j  T_i   
	+  f_{\sigma_i\rho_j\sigma_i(a)}  T_i^2
	+ f_{\sigma_i\sigma_j\rho_i(a)} T_i  T_j  
	\\
\notag&	+ (f_{\rho_i\rho_j\sigma_i(a)} + f_{\sigma_i\rho_j\rho_i(a)}) T_i  
	+ f_{\rho_i\sigma_j\rho_i(a)}  T_j  
	+ f_{\rho_i\rho_j\rho_i(a)} 
	\\
\notag&\quad= f_{\sigma_i\sigma_j\sigma_i(a)} T_i  T_j  T_i   
	+ f_{\rho_i\sigma_j\sigma_i(a)} T_j  T_i   
	+ f_{\sigma_i\sigma_j\rho_i(a)} T_i  T_j
	+ (f_{\sigma_i\rho_j\sigma_i(a) S_i} + f_{\rho_i\rho_j\sigma_i(a)} + f_{\sigma_i\rho_j\rho_i(a)}) T_i  \\
\notag&\quad\quad
	+ f_{\rho_i\sigma_j\rho_i(a)}  T_j  
	+ (f_{\rho_i\rho_j\rho_i(a)} + f_{\sigma_i\rho_j\sigma_i(a) R_i})	\ \ \ \text{ by \hyperlink{eq:Ql}{$Q[\ell - 3]$}, \hyperlink{eq:Ml}{$M[\ell - 3]$}} \\
\notag&\quad= f_{\sigma_j\sigma_i\sigma_j(a)} T_i  T_j  T_i   
	+ f_{\rho_j\sigma_i\sigma_j(a)} T_i  T_j   
	+ f_{\sigma_j\sigma_i\rho_j(a)} T_j  T_i
	+ (f_{\sigma_j\rho_i\sigma_j(a) S_j} + f_{\rho_j\rho_i\sigma_j(a)} + f_{\sigma_j\rho_i\rho_j(a)}) T_j  \\
\notag&\quad\quad
	+ f_{\rho_j\sigma_i\rho_j(a)}  T_i  
	+ (f_{\rho_j\rho_i\rho_j(a)} + f_{\sigma_j\rho_i\sigma_j(a) R_j}) \ \ \ \text{ by \eqref{def:br1}, \eqref{def:br2}, \eqref{def:br3}} \\
\notag&\quad=
	f_{\sigma_j\sigma_i\sigma_j(a)} T_j  T_i  T_j   
	+ f_{\rho_j\sigma_i\sigma_j(a)} T_i  T_j   
	+  f_{\sigma_j\rho_i\sigma_j(a)}  T_j^2
	+ f_{\sigma_j\sigma_i\rho_j(a)} T_j  T_i  \\
\notag&\quad\quad
	+ f_{\rho_j\sigma_i\rho_j(a)}  T_i  
	+ (f_{\rho_j\rho_i\rho_j(a)} + f_{\sigma_j\rho_i\sigma_j(a) R_j}) \ \ \ \text{ by \hyperlink{eq:B3l}{$B_3[\ell - 3]$}} \\
\notag&\quad=
	f_{\sigma_j\sigma_i\sigma_j(a)} T_j  T_i  T_j   
	+ f_{\rho_j\sigma_i\sigma_j(a)} T_i  T_j   
	+  f_{\sigma_j\rho_i\sigma_j(a)}  T_j^2
	+ f_{\sigma_j\sigma_i\rho_j(a)} T_j  T_i  \\
\notag&\quad\quad
	+ (f_{\rho_j\rho_i\sigma_j(a)} + f_{\sigma_j\rho_i\rho_j(a)}) T_j  
	+ f_{\rho_j\sigma_i\rho_j(a)}  T_i  
	+ f_{\rho_j\rho_i\rho_k(a)}.	
	\ \ \ \text{ by \hyperlink{eq:Ql}{$Q[\ell - 3]$}, \hyperlink{eq:Ml}{$M[\ell - 3]$}} 
\end{align}
\end{allowdisplaybreaks}
\end{proof}

\end{document}